\documentclass[11pt,a4paper,reqno]{amsart}

\usepackage[utf8]{inputenc}
\usepackage[english]{babel}
\usepackage[T1]{fontenc}
\usepackage{lmodern}
\usepackage{color}

\usepackage[a4paper,left=35mm,right=35mm,top=35mm,bottom=35mm,marginpar=25mm]{geometry}

\usepackage[spacing=true,
			kerning=true,
			babel=true,
			tracking=true]{microtype}
\SetTracking{encoding={*}, shape=sc}{0}

\usepackage{amsmath,amsthm,amssymb,amsfonts}

\RequirePackage[unicode]{hyperref}
\hypersetup{
	bookmarksopen,
	bookmarksopenlevel=1,
	pdfborder={0 0 0},
	pdfdisplaydoctitle,
	pdfpagelayout={SinglePage},
	pdfstartview={FitH},
	pdfauthor={Kamil Kosiba (University of Warwick) and Filip Rindler (University of Warwick)},
	pdftitle={On the relaxation of integral functionals depending on the symmetrized gradient}
}

\numberwithin{equation}{section}

\newtheoremstyle{thmlemcorr}{10pt}{10pt}{\itshape}{}{\bfseries}{.}{10pt}{{\thmname{#1}\thmnumber{ #2}\thmnote{ (#3)}}}
\newtheoremstyle{thmlemcorr*}{10pt}{10pt}{\itshape}{}{\bfseries}{.}\newline{{\thmname{#1}\thmnumber{ #2}\thmnote{ (#3)}}}
\newtheoremstyle{remexample}{10pt}{10pt}{}{}{\bfseries}{.}{10pt}{{\thmname{#1}\thmnumber{ #2}\thmnote{ (#3)}}}

\theoremstyle{thmlemcorr}
\newtheorem{theorem}{Theorem}
\numberwithin{theorem}{section}
\newtheorem{lemma}[theorem]{Lemma}
\newtheorem{corollary}[theorem]{Corollary}
\newtheorem{proposition}[theorem]{Proposition}

\newtheorem{definition}[theorem]{Definition}

\theoremstyle{thmlemcorr*}
\newtheorem{theorem*}{Theorem}
\newtheorem{lemma*}[theorem]{Lemma}
\newtheorem{corollary*}[theorem]{Corollary}
\newtheorem{proposition*}[theorem]{Proposition}
\newtheorem{problem*}[theorem]{Problem}
\newtheorem{conjecture*}[theorem]{Conjecture}
\newtheorem{definition*}[theorem]{Definition}

\theoremstyle{remexample}
\newtheorem{remark}[theorem]{Remark}

\DeclareMathOperator{\id}{Id}
\newcommand{\dif}{\mathrm{d}}

\let\div\relax
\DeclareMathOperator{\div}{div}
\DeclareMathOperator{\dist}{dist}

\DeclareMathOperator{\tr}{tr}
\DeclareMathOperator{\supp}{supp}
\DeclareMathOperator{\dev}{dev}

\DeclareMathOperator{\LL}{\mathrm{L}}
\DeclareMathOperator{\WW}{\mathrm{W}}
\DeclareMathOperator{\CC}{\mathrm{C}}
\DeclareMathOperator{\BV}{\mathrm{BV}}
\DeclareMathOperator{\BD}{\mathrm{BD}}
\DeclareMathOperator{\LD}{\mathrm{LD}}
\DeclareMathOperator{\UU}{\mathrm{U}}
\DeclareMathOperator{\LU}{\mathrm{LU}}
\DeclareMathOperator{\M}{\mathrm{M}}
\DeclareMathOperator{\Mpos}{\mathrm{M}^{+}}
\DeclareMathOperator{\Mloc}{\mathrm{M}_{\mathrm{loc}}}
\DeclareMathOperator{\MlocProb}{\mathrm{M}^1_{\mathrm{loc}}}
\DeclareMathOperator{\MlocPos}{\mathrm{M}^{+}_{\mathrm{loc}}}

\DeclareMathOperator{\SD}{SD({\it d})}

\newcommand{\starconv}{\stackrel{_\ast}{\rightharpoondown}}

\usepackage{mathtools}

\DeclarePairedDelimiter\norm{\lVert}{\rVert}

\usepackage{bbm}

\usepackage{bm}

\newcommand{\mres}{\mathbin{\vrule height 1.6ex depth 0pt width
		0.13ex\vrule height 0.13ex depth 0pt width 1.3ex}}

\makeatletter
\renewcommand*{\eqref}[1]{%
	\hyperref[{#1}]{\textup{\tagform@{\ref*{#1}}}}%
}
\makeatother

\makeatletter
\newcommand*{\mint}[1]{%
	\mint@l{#1}{}%
}
\newcommand*{\mint@l}[2]{%
	\@ifnextchar\limits{%
		\mint@l{#1}%
	}{%
		\@ifnextchar\nolimits{%
			\mint@l{#1}%
		}{%
			\@ifnextchar\displaylimits{%
				\mint@l{#1}%
			}{%
				\mint@s{#2}{#1}%
			}%
		}%
	}%
}
\newcommand*{\mint@s}[2]{%
	\@ifnextchar_{%
		\mint@sub{#1}{#2}%
	}{%
		\@ifnextchar^{%
			\mint@sup{#1}{#2}%
		}{%
			\mint@{#1}{#2}{}{}%
		}%
	}%
}
\def\mint@sub#1#2_#3{%
	\@ifnextchar^{%
		\mint@sub@sup{#1}{#2}{#3}%
	}{%
		\mint@{#1}{#2}{#3}{}%
	}%
}
\def\mint@sup#1#2^#3{%
	\@ifnextchar_{%
		\mint@sub@sup{#1}{#2}{#3}%
	}{%
		\mint@{#1}{#2}{}{#3}%
	}%
}
\def\mint@sub@sup#1#2#3^#4{%
	\mint@{#1}{#2}{#3}{#4}%
}
\def\mint@sup@sub#1#2#3_#4{%
	\mint@{#1}{#2}{#4}{#3}%
}
\newcommand*{\mint@}[4]{%
	\mathop{}%
	\mkern-\thinmuskip
	\mathchoice{%
		\mint@@{#1}{#2}{#3}{#4}%
		\displaystyle\textstyle\scriptstyle
	}{%
		\mint@@{#1}{#2}{#3}{#4}%
		\textstyle\scriptstyle\scriptstyle
	}{%
		\mint@@{#1}{#2}{#3}{#4}%
		\scriptstyle\scriptscriptstyle\scriptscriptstyle
	}{%
		\mint@@{#1}{#2}{#3}{#4}%
		\scriptscriptstyle\scriptscriptstyle\scriptscriptstyle
	}%
	\mkern-\thinmuskip
	\int#1%
	\ifx\\#3\\\else_{#3}\fi
	\ifx\\#4\\\else^{#4}\fi
}
\newcommand*{\mint@@}[7]{%
	\begingroup
	\sbox0{$#5\int\m@th$}%
	\sbox2{$#5\int_{}\m@th$}%
	\dimen2=\wd0 %
	\let\mint@limits=#1\relax
	\ifx\mint@limits\relax
	\sbox4{$#5\int_{\kern1sp}^{\kern1sp}\m@th$}%
	\ifdim\wd4>\wd2 %
	\let\mint@limits=\nolimits
	\else
	\let\mint@limits=\limits
	\fi
	\fi
	\ifx\mint@limits\displaylimits
	\ifx#5\displaystyle
	\let\mint@limits=\limits
	\fi
	\fi
	\ifx\mint@limits\limits
	\sbox0{$#7#3\m@th$}%
	\sbox2{$#7#4\m@th$}%
	\ifdim\wd0>\dimen2 %
	\dimen2=\wd0 %
	\fi
	\ifdim\wd2>\dimen2 %
	\dimen2=\wd2 %
	\fi
	\fi
	\rlap{%
		$#5%
		\vcenter{%
			\hbox to\dimen2{%
				\hss
				$#6{#2}\m@th$%
				\hss
			}%
		}%
		$%
	}%
	\endgroup
}

\begin{document}
\title[On the relaxation of integral functionals]{On the relaxation of integral functionals depending on the symmetrized gradient}
\date{\today}

\author[K. Kosiba]{Kamil Kosiba}
\address{Mathematics Institute, University of Warwick, Coventry CV4 7AL, UK}
\email{K.Kosiba@warwick.ac.uk}

\author[F. Rindler]{Filip Rindler}
\address{Mathematics Institute, University of Warwick, Coventry CV4 7AL, UK, and The Alan Turing Institute, British Library, 96 Euston Road, London NW1 2DB, UK}
\email{F.Rindler@warwick.ac.uk}


\begin{abstract}
	We prove results on the relaxation and weak* lower semicontinuity of integral functionals of the form
	\[
	\mathcal{F}[u] :=
	\int_{\Omega}f \Bigl( \frac{1}{2} \bigl( \nabla u(x) + \nabla u(x)^T \bigr) \Bigr)\,\dif x,  \qquad u \colon \Omega \subset \mathbb{R}^d \to \mathbb{R}^d,
	\]
	over the space $\BD(\Omega)$ of functions of bounded deformation or over the Temam--Strang space
	\[
	\UU(\Omega):=\bigl\{u\in\BD(\Omega): \ \div u\in\LL^2(\Omega)\bigr\},
	\]
	depending on the growth and shape of the integrand $f$. Such functionals are interesting for example in the study of Hencky plasticity and related models.

\end{abstract}

\maketitle

\section{Introduction}\label{sec:introduction}
Let $\Omega\subset\mathbb{R}^d$, $d\geq 2$ be a bounded Lipschitz domain occupied by some elasto-plastic material body and let $u:\Omega\to\mathbb{R}^d$ denote a displacement field. The classical minimization problem in the theory of Hencky plasticity~\cite{Temam85book,AnzellottiGiaquinta80, FuchsSeregin00} involves the following convex functional:
\begin{equation}\label{eq:hencky-intro}
\int_{\Omega}\varphi(\dev\mathcal{E}u)+\frac{\kappa}{2}(\div u)^2\,\dif x,
\end{equation}
where $\varphi:\SD\to[0,+\infty)$ is a function which grows quadratically on some compact set and linearly outside of this set, and $\kappa=\lambda+2\mu/3 > 0$ is the bulk modulus of the material with the Lam{\'e} constants $ \lambda $ and $ \mu $. Here, $\SD$ denotes the space of symmetric and deviatoric matrices in $\mathbb{R}^{d\times d}$ and $\dev A:=A-d^{-1}(\tr A)\id$ is the deviatoric (trace-free) part of a~matrix $A\in\mathbb{R}^{d\times d}$. We also write $\mathcal{E}u$ for the symmetrized gradient, i.e.
\[
\mathcal{E}u := \frac{1}{2}\left(\nabla u+\nabla u^T\right).
\]
The minimization problem~\eqref{eq:hencky-intro} and its relaxation have attracted much attention recently. For instance, in~\cite{BraidesDefranceschiVitali97} the authors studied the same problem with an additional jump penalization term. In~\cite{Mora16} the $\LL^1$-relaxation of~\eqref{eq:hencky-intro} is identified, further generalized in~\cite{JesenkoSchmidt17} to allow integrands for which deviatoric and trace components are not necessarily separated additively. In~\cite{Bojarski97part1, Bojarski97part2} the author investigates the relaxation of Signorini problems in the framework of Hencky's plasticity.

Here we consider the functional~\eqref{eq:hencky-intro} to be generalized to include possibly non-convex integrands,  i.e., we consider functionals of the form
\begin{equation}\label{eq:generalized-intro}
\mathcal{F}[u,\Omega]:=\int_{\Omega}f(\mathcal{E}u)\,\dif x,
\end{equation}
where the continuous integrand $f:\mathbb{R}^{d\times d}_{\mathrm{sym}}\to[0,+\infty)$ satisfies the anisotropic growth conditions
\begin{equation}\label{eq:growth-intro}
m\left((\tr A)^2+|\dev A|\right)
\leq
f(A)
\leq
M\left(1+(\tr A)^2+|\dev A|\right)
\end{equation}
for some constants $0<m\leq M$ and all $A \in \mathbb{R}^{d\times d}_{\mathrm{sym}}$.

A first choice for a function space on which to define~\eqref{eq:generalized-intro} with~\eqref{eq:growth-intro} is the space of integrable functions $u$ with integrable symmetrized distributional derivative $\mathcal{E}u$ and square-integrable distributional divergence, i.e.\
\[
\LU(\Omega)
:=
\left\{
u\in\LL^1(\Omega;\mathbb{R}^d): \ \mathcal{E}u\in\LL^1(\Omega;\mathbb{R}^{d\times d}_{\mathrm{sym}}), \ \div u\in\LL^2(\Omega)
\right\}.
\]
Unfortunately, in this space the direct method of the calculus of variations does not provide any solution to the minimization problem. The culprit is the lack of reflexivity and consequently, the inability to infer the (weak) relative compactness from the norm-boundedness of a minimising sequence.

Therefore, the functional~\eqref{eq:generalized-intro} needs to be extended to account for displacement fields $u$ whose linear strains $Eu$ are measures, since in the space of measures norm-boundedness of minimising sequence implies weak* relative compactness. Then the usual direct method applies. For this, one first introduces the space $\BD(\Omega)$ of \emph{functions of bounded deformation} as the space of all functions $u\in\LL^1(\Omega;\mathbb{R}^d)$ such that the distributional symmetrized derivative $Eu:=\frac{1}{2}(Du+Du^T)$ is representable as a finite Radon measure $Eu\in\M(\Omega;\mathbb{R}^{d\times d}_{\mathrm{sym}})$. Then, the \emph{Temam--Strang space} $\UU(\Omega)$ is the space of functions of bounded deformation with square-integrable divergence, i.e.\
\[
\UU(\Omega)
:=
\left\{
u\in\BD(\Omega): \ \div u\in\LL^2(\Omega)
\right\}.
\]
For more information on $\BD, \UU$ and their applications in the theory of plasticity we refer to~\cite{AmbrosioCosciaDalMaso97, FuchsSeregin00, Kohn82, MatthiesStrangChristiansen79book, Suquet78, Suquet79, TemamStrang80, Temam85book}.

For an integrand $f$ that is additionally symmetric rank-one convex (see below), it was then established in~\cite{JesenkoSchmidt17} that the `continuity extension' of the functional~\eqref{eq:generalized-intro} over the Temam--Strang space is given by
\begin{equation}\label{eq:extension-intro}
\overline{\mathcal{F}}[u,\Omega]
:=
\int_{\Omega}f(\mathcal{E}u)\,\dif x
+
\int_{\Omega}f_{\dev}^{\#}\left(\frac{\dif E^su}{\dif|E^su|}\right)\,\dif|E^su|, \quad u\in\UU(\Omega);
\end{equation}
see Theorem~\ref{thm:jesenko-schmidt} for details. Here, the strain $Eu$ is decomposed into $Eu=E^au+E^su=\mathcal{E}u\mathcal{L}^d\mres\Omega+E^su$ according to the Lebesgue decomposition theorem, $\frac{\dif E^su}{\dif|E^su|}$ is the polar density of the singular part $E^su$ with respect to $|E^su|$, and $f_{\dev}^{\#}$ is the upper recession function of the restriction $f_{\dev}$ of $f$ to the symmetric deviatoric $(d \times d)$-matrices, denoted by $\SD$, i.e.
\begin{equation}\label{eq:recession-intro}
f_{\dev}^{\#}(D)
:=
\limsup\limits_{\substack{D'\to D \\ s\to\infty}}\frac{f_{\dev}(sD')}{s}, \quad D\in\SD.
\end{equation}

As our first result, we extend the previous strong $\LL^1$-relaxation result by Jesenko and Schmidt (see Proposition~3.15 in~\cite{JesenkoSchmidt17}) to a weak* relaxation in the Temam-Strang space with a weaker subcritical lower bound on the integrand. For this, we define the relaxation $\mathcal{F}_{*}$ of $\mathcal{F}$ for $u \in \UU(\Omega)$ as follows:
\begin{equation} \label{eq:relaxU}
\mathcal{F}_{*}[u,\Omega]
:=
\inf\bigg\{ \liminf_{h\to\infty}\mathcal{F}[u_h,\Omega]: \ (u_h)\subset\LU(\Omega), \
u_h\starconv u \text{ in } \UU(\Omega) \bigg\},
\end{equation}
where the weak* convergence is understood in a sense of Definition \ref{def:weak-conv} below.

\newpage

\begin{theorem}\label{thm:main-result}
	Let $\Omega\subset\mathbb{R}^d$ be a bounded Lipschitz domain and let $f:\mathbb{R}^{d\times d}_{\mathrm{sym}}\to[0,\infty)$ be a continuous function satisfying the following conditions:
	\begin{enumerate}
		\item there exist constants $0< m\leq M$ such that for all $A\in\mathbb{R}^{d\times d}_{\mathrm{sym}}$ the growth
		\begin{equation}\label{eq:growth}
		m\left((\tr A)^2+|\dev A|\right)
		\leq
		f(A)
		\leq
		M\left(1+(\tr A)^2+|\dev A|\right)
		\end{equation}
		holds;
		\item $f$ is symmetric-quasiconvex, that is for any bounded Lipschitz domain $\omega\subset\mathbb{R}^d$, any symmetric matrix $A\in\mathbb{R}^{d\times d}_{\mathrm{sym}}$ and any $\psi\in\WW^{1,\infty}_0(\omega;\mathbb{R}^d)$ the inequality
		\[
		|\omega|f(A)\leq\int_{\omega}f(A+\mathcal{E}\psi(y))\,\dif y
		\]
		holds;
		\item\label{enum:lower-bound} there exist constants $\gamma\in[0,2)$ and $\delta\in[0,1)$ such that for all $A\in\mathbb{R}^{d\times d}_{\mathrm{sym}}$ the inequality
		\begin{equation}\label{eq:subcritical-growth}
		f(A)\geq f_{\dev}^{\#}(\dev A)-M\left(|\tr A|^\gamma+|\dev A|^\delta + 1\right)
		\end{equation}
		holds.
	\end{enumerate}
	Then, the functional~\eqref{eq:extension-intro} is the relaxation of~\eqref{eq:generalized-intro} with respect to weak* topology in $\UU(\Omega)$, that is $\mathcal{F}_{*}[u,\Omega] = \overline{\mathcal{F}}[u,\Omega]$.
\end{theorem}

\begin{remark}
	The lower bound with subcritical growth in both trace and deviatoric directions in the condition (3) is essential for the proof. It remains an open question whether it can be deduced from the conditions (1) and (2).
\end{remark}

It does not seem possible to prove Theorem~\ref{thm:main-result} using the blow-up argument for both regular and singular estimates as in the usual $\BV$ lower semicontinuity results~\cite{FonsecaMuller93,AmbrosioDalMaso92}. The classical blow-up argument was tailored for functionals with an isotropic linear growth imposed on the integrands. This, however, is not the case here, as the admissible integrands in Theorem~\ref{thm:main-result} grow quadratically in the trace direction. The problem is that if one attempts to utilise the blow-up argument at singular points, one eventually faces the problem of controlling the blow-up rate of the divergence terms of the blow-up sequence. A priori it seems not possible to obtain a sufficient decay of this sequence of divergences, and so a different strategy based on the Kirchheim--Kristensen convexity result~\cite{KirchheimKristensen11,KirchheimKristensen16} needs to be employed.

As our second result, we  give  a refined relaxation theorem in $\BD$ for homogeneous integrands, improving the results of~\cite{Rindler11,BarrosoFonsecaToader00,A-RDePR2017} to an essentially optimal (under the following growth conditions) result:

\begin{theorem}\label{thm:lsc-bd}
	Let $\Omega\subset\mathbb{R}^d$ be a bounded Lipschitz domain and let $f:\mathbb{R}^{d\times d}_{\mathrm{sym}}\to[0,\infty)$ be a continuous function such that there exist constants $0<m\leq M$, for which the inequality
	\begin{equation} \label{eq:growth-bd}
	m|A|\leq f(A)\leq M(1+|A|), \quad A\in\mathbb{R}^{d\times d}_{\mathrm{sym}},
	\end{equation}
	holds.
	Then, the functional
	\begin{align*}
	\overline{\mathcal{F}}[u,\Omega]
	&:=
	\int_{\Omega}(SQf)(\mathcal{E}u)\,\dif x
	+
	\int_{\Omega}(SQf)^{\#}\left(\frac{\dif E^su}{\dif |E^su|}\right)\,\dif|E^su|, \quad u\in\BD(\Omega)
	\end{align*}
	is the relaxation of the functional
	\begin{align*}
	\mathcal{F}[u,\Omega]
	&:=
	\int_{\Omega}f(\mathcal{E}u)\,\dif x, \quad u\in\LD(\Omega):=\BD(\Omega)\cap\{u: \ E^su=0\}
	\end{align*}
	with respect to the weak* topology in $\BD(\Omega)$.
\end{theorem}

Here, the relaxation $\mathcal{F}_{*}$ of $\mathcal{F}$ is defined as
\begin{equation} \label{eq:relaxBD}
\mathcal{F}_{*}[u,\Omega]:=\inf\left\{\liminf_{h\to\infty} \mathcal{F}[u_h,\Omega]: \ (u_h)\subset\LD(\Omega), \ u_h\starconv u \text{ in } \BD(\Omega)\right\}.
\end{equation}
Moreover, $SQf:\mathbb{R}^{d\times d}_{\mathrm{sym}}\to[0,\infty)$ is the \emph{symmetric-quasiconvex envelope}, defined by
\begin{equation*}
SQf(A)
:=
\inf\left\{ \mint{-}_{D}f(A+\mathcal{E}\psi(y))\,\dif y: \ \psi\in\WW^{1,\infty}_0(D;\mathbb{R}^d) \right\}.
\end{equation*}
The set $D\subset\mathbb{R}^d$ in the above formula is an arbitrary bounded Lipschitz domain.

In Theorem~1.1 in~\cite{Rindler11}, only a weak* lower semicontinuity result, and not a full relaxation result, was established under the assumption that the strong recession function $f^{\infty}$ exists. Our Theorem~\ref{thm:lsc-bd} extends~\cite{BarrosoFonsecaToader00} and also Corollary~1.10 in~\cite{A-RDePR2017} to a relaxation theorem without any assumption on the recession function. We note that in view of Theorem~2 in~\cite{Muller92}, one can construct a function satisfying~\eqref{eq:growth-bd}, for which $f^{\infty}$ does not exist.

\begin{remark}
	In Theorems~\ref{thm:main-result} and \ref{thm:lsc-bd} the weak upper recession functions $f_{\mathrm{dev}}^{\#}$ and $(SQf)^{\#}$ respectively are evaluated only on matrices from the symmetric rank-one cone. As the integrands $f$ and $SQf$ are symmetric-quasiconvex, hence symmetric rank-one convex, by a simple convexity argument one can show that the upper recession functions agree with the lower recession functions (with the lower limit in place of the upper limit). This means that in the statements of Theorems~\ref{thm:main-result} and \ref{thm:lsc-bd} we could use the lower recession functions instead (which, in a sense, are more natural for lower semicontinuity results). However, while $(SQf)^{\#}$ can easily be seen to be symmetric quasiconvex (by Fatou's lemma), this is not possible for the lower recession function; see also Remarks~8,~9 in~\cite{Rindler11} for further discussion of this subtle point.
\end{remark}

\section*{Acknowledgements}
The authors would like to thank Martin Jesenko, Bernd Schmidt, Jan Kristensen and Wojciech O{\.z}a{\'n}ski for many fruitful discussions related to this work.

This project has received funding from the European Research Council (ERC) under the European Union's Horizon 2020 research and innovation programme, grant agreement No 757254 (SINGULARITY). K.\ K.\ also gratefully acknowledges the financial support provided by the EPSRC as a part of the MASDOC DTC at the University of Warwick (EP/HO23364/1).

\section{Preliminaries}\label{sec:prelims}
By $\mathbb{R}^d$ we denote the $d$-dimensional Euclidean space with $d\geq 1$. We write $B(x,r)$ for an open ball, $\overline{B(x,r)}$ for a closed ball and $\partial B(x,r)$ for a sphere centered at $x\in\mathbb{R}^d$ with the radius $r>0$. For any matrix $A\in\mathbb{R}^{d\times d}$ its \emph{deviatoric projection} is defined as $\dev A:=A-d^{-1}(\tr A)\id$, where $\id\in\mathbb{R}^{d\times d}$ is the identity matrix. The set of all symmetric and deviatoric matrices in $\mathbb{R}^{d\times d}$ is denoted by
\[
\SD:=\{M\in\mathbb{R}^{d\times d}_{\mathrm{sym}}: \ \tr M=0 \}.
\]

In this paper we always assume that $\Omega\subset\mathbb{R}^d$ is an open bounded Lipschitz domain, unless stated otherwise.

We write $\LL^p(\Omega)$, $\LL^p(\Omega;X)$, $\LL^p_{loc}(\Omega)$, etc. for the Lebesgue spaces and $\WW^{k,p}(\Omega)$, $\WW^{k,p}(\Omega;X)$, $\WW^{k,p}_g(\Omega)$, etc. for the Sobolev spaces with suitable exponents.

\subsection{Measure theory}\label{subsec:measure-theory}
We write $\mathcal{B}(X)$ for the Borel $\sigma$-algebra on a topological space $X$. The $d$-dimensional Lebesgue measure is denoted by $\mathcal{L}^d$ and for the $\mathcal{L}^d$-measurable set $A\subseteq\mathbb{R}^d$ we occasionally write $|A|$ instead of $\mathcal{L}^d(A)$.

The cone of (finite) Radon measures is denoted by $\M^+(\mathbb{R}^d)$ and its subspace of probability measures is denoted by $\M^1(\mathbb{R}^d)$. We also use local versions of these spaces denoted by $\MlocPos(\mathbb{R}^d)$ and $\MlocProb(\mathbb{R}^d)$, where the measures restricted to any compact set $K\subset\mathbb{R}^d $ are in $\M^+(K)$ or $\M^1(K)$, respectively.

The following theorem provides a simple criterion for a set function to be a Radon measure (for the proof see Theorem~1.53 in~\cite{AmbrosioFuscoPallara00}).
\begin{theorem}[De Giorgi-Letta]\label{thm:degiorgi-letta}
	Let $X$ be a metric space and let $\mathcal{U}(X)$ denote the set of open subsets of $X$. Let $\mu:\mathcal{U}(X)\to[0,\infty]$ be a set function such that
	\begin{enumerate}
		\item $\mu(\emptyset)=0$;
		\item (monotonicity) for $ A,B\in\mathcal{U}(X) $ if $A\subset B$ then $\mu(A)\leq\mu(B)$;
		\item (subadditivity) for $ A,B\in\mathcal{U}(X) $ it holds that $ \mu(A\cup B)\leq\mu(A)+\mu(B) $;
		\item (superadditivity) for $ A,B\in\mathcal{U}(X) $ with $A\cap B=\emptyset$ it holds that $ \mu(A\cup B)\geq\mu(A)+\mu(B) $;
		\item (inner regularity) $\mu(A)=\sup\left\{\mu(B): \ B\in\mathcal{U}(X), \ B\Subset A \right\}$.
	\end{enumerate}
	Then, the extension of $\mu$ to every $B\subset X$ defined by
	\[
	\mu(B):=\inf\left\{ \mu(A): \ A\in\mathcal{U}(X), \ A\supset B\right\}
	\]
	is an outer measure. In particular, the restriction of $\mu$ to the Borel $ \sigma $-algebra is a positive measure.
\end{theorem}

Let $\mu$ be a positive Radon measure in an open set $\Omega\subset\mathbb{R}^d$ and let $k\geq 0$. We define the \emph{upper $k$-density} of $\mu$ at $x\in\Omega$ as
\[
\Theta^*_k(\mu,x):=\limsup_{r\downarrow 0}\frac{\mu(B(x,r))}{\omega_kr^k},
\]
where $\omega_k:=\pi^{k/2}\Gamma(1+k/2)$ is the Lebesgue measure of a unit ball in $\mathbb{R}^k$.

The following result (see Theorem~2.56 in~\cite{AmbrosioFuscoPallara00} for the proof) asserts that the upper $k$-density can be used to estimate the measure $\mu$ from below by the $k$-dimensional Hausdorff measure $\mathcal{H}^k$.
\begin{proposition}\label{prop:lowerHausdorffBound}
	Let $\Omega\subset\mathbb{R}^d$ be an open set and let $\mu$ be a positive Radon measure in $\Omega$. Then, for any $0<t<\infty$ and any Borel set $B\subset\Omega$ the implication
	\[
	\Theta^*_k(\mu,x)\geq t \quad \forall\; x\in B
	\implies
	\mu\geq t\mathcal{H}^k\mres B
	\]
	holds.
\end{proposition}

We also use vector-valued Borel measures $\mu:\mathcal{B}(\mathbb{R}^d)\to\mathbb{R}^N$, which are $\sigma$-additive set functions with $\mu(\emptyset)=0$. The space of all such vector measures is denoted by $\M(\mathbb{R}^d;\mathbb{R}^N)$. The space of local vector measures is denoted by $\Mloc(\mathbb{R}^d;\mathbb{R}^N)$. For a vector measure $\mu\in\M(\mathbb{R}^d;\mathbb{R}^N)$ we define its \emph{total variation measure} $|\mu|\in\M^+(\mathbb{R}^d)$ for every Borel set $S \subset \Omega$ by
\[
|\mu|(S)
:=
\sup\left\{ \sum_{k\in\mathbb{N}}|\mu(S_k)|: \ S=\bigcup_{k\in\mathbb{N}}S_k, \ \{S_k\} \text{ is a Borel partition of } S \right\}.
\]
The \emph{restriction} of a measure $\mu\in\Mloc(\mathbb{R}^d;\mathbb{R}^N)$ to a Borel set $B\in\mathcal{B}(\mathbb{R}^d)$ is defined as $\mu\mres B(S):=\mu(B\cap S)$ for all relatively compact Borel sets $S\in\mathcal{B}(\mathbb{R}^d)$.

For a positive measure $\mu$ on a locally compact separable metric space $X$, the \emph{support} of $\mu$, in symbols $\supp \mu$, is the closed set  of all points $x\in X$ such that $ \mu(U)>0 $ for every neighbourhood $U$ of $x$. For a vector measure $\nu$ we define its support to be the support of its total variation measure $|\nu|$.

\begin{theorem}[Besicovitch differentiation theorem]
	\label{thm:besicovitch}
	Let $\mu\in\M(\mathbb{R}^d;\mathbb{R}^N)$ be a vector-valued Radon measure and let $\nu\in\M^+(\mathbb{R}^d)$ be a positive Radon measure. Then for $\nu$-a.e.\ $x_0\in\mathbb{R}^d$ in the support of $\nu$, the limit
	\[
	\frac{\dif\mu}{\dif\nu}(x_0):=\lim_{r\downarrow 0}\frac{\mu(B(x_0,r))}{\nu(B(x_0,r))}
	\]
	exists and is called the \emph{Radon-Nikodym derivative} of $ \mu $ with respect to $ \nu $.

	Moreover, we have the Lebesgue decomposition of $\mu=\frac{\dif\mu}{\dif\nu}\nu+\mu^s$, where $\mu^s=\mu\mres E$ is singular with respect to $ \nu $ and
	\[
	E=(\mathbb{R}^d\setminus\supp\nu)\cup\left\{x\in\supp\nu: \ \lim_{r\downarrow 0}\frac{|\mu|(B(x,r))}{\nu(B(x,r))}=\infty \right\}.
	\]
\end{theorem}
For the proof, see Theorem~2.22 in \cite{AmbrosioFuscoPallara00}. See also Theorem~5.52 in \cite{AmbrosioFuscoPallara00} for a more general version, where a ball $B(x_0,r)$ can be replaced with a set $x_0+rC$ for any open convex set $C\subset\mathbb{R}^d$ containing the origin.

\subsection{Function spaces}\label{subsec:function-spaces}
In this section we briefly recall definitions and basic properties of the space of functions with bounded deformation and its subspace with square-integrable distributional divergences, called the Temam--Strang space.

\subsubsection{Functions of bounded deformation.}\label{subsubsec:space-BD}
In applications coming from plasticity theory~\cite{Suquet78,Suquet79,TemamStrang80} one is often concerned with the class of functions
\[
\LD(\Omega):=\left\{ u\in\LL^1(\Omega;\mathbb{R}^d): \ \mathcal{E}u\in\LL^1(\Omega;\mathbb{R}^{d\times d}_{\mathrm{sym}}) \right\},
\]
where $\mathcal{E}u:=\frac{1}{2}(\nabla u+\nabla u^T)$ is the distributional symmetrized gradient of a mapping $u:\Omega\to\mathbb{R}^d$. The space $\LD(\Omega)$ is a Banach space when endowed with the norm
\[
\norm{u}_{\LD}:=\norm{u}_{\LL^1}+\norm{\mathcal{E}u}_{\LL^1}.
\]
However, in general we cannot infer weak relative compactness from boundedness, since $\LD(\Omega)$ is not reflexive. If a bounded sequence in $\LD(\Omega)$ has equiintegrable symmetric gradients, then in virtue of the Dunford--Pettis theorem, we could infer the weak relative compactness. The equiintegrability, however, is rare in applications, so we need to consider a larger space instead.

Therefore, we define the space $\BD(\Omega)$ of \emph{functions of bounded deformation}~\cite{Suquet78,Suquet79,TemamStrang80,AmbrosioCosciaDalMaso97} as the space of all functions $u\in\LL^1(\Omega;\mathbb{R}^d)$ such that the distributional symmetrized derivative is representable as a finite Radon measure $Eu\in\M(\Omega;\mathbb{R}^{d\times d}_{\mathrm{sym}})$. The space $\BD(\Omega)$ is a Banach space when endowed with the norm
\[
\norm{u}_{\BD}:=\norm{u}_{\LL^1}+|Eu|(\Omega).
\]
According to the Lebesgue decomposition theorem, we split the measure $Eu$ into
\[
Eu=\mathcal{E}u\mathcal{L}^d+E^su,
\]
where $\mathcal{E}u:=\frac{\dif Eu}{\dif\mathcal{L}^d}\in\LL^1(\Omega,\mathcal{L}^d;\mathbb{R}^{d\times d}_{\mathrm{sym}})$ is the Radon-Nikodym derivative of $Eu$ with respect to the Lebesgue measure $\mathcal{L}^d$ (called the \emph{approximate symmetrized gradient}) and $E^su\perp\mathcal{L}^d$ is the \emph{singular part} of $Eu$.

We have the following $\BD$-analogue of Alberti's rank-one theorem in $\BV$ (cf.~\cite{Alberti93,MassaccesiVittone16}).
\begin{theorem}\label{thm:bd-rank-one}
	Let $\Omega\subset\mathbb{R}^d$ be an open set and let $u\in\BD(\Omega)$. Then, for $|E^su|$-a.e.\ $ x\in\Omega $, there exist $a(x),b(x)\in\mathbb{R}^d\setminus\{0\}$ such that
	\[
	\frac{\dif E^su}{\dif|E^su|}(x)=a(x)\odot b(x),
	\]
	where $ a\odot b:=(a\otimes b+b\otimes a)/2 $ denotes the symmetrized tensor product.
\end{theorem}
For the proof, see~\cite{DePhilippisRindler16}.

\subsubsection{Temam--Strang space.}\label{subsubsec:space-T-S}
For the theory of elasto-plasticity in the geometrically linear setting the class of functions defined as
\[
\LU(\Omega)
:=
\left\{ u\in\LD(\Omega): \ \div u\in\LL^2(\Omega) \right\}
\]
becomes a natural choice~\cite{FuchsSeregin00,Temam85book,JesenkoSchmidt17,ContiOrtiz05}. Unfortunately, the space $\LU(\Omega)$ inherits the poor compactness property of $\LD(\Omega)$ and again, it is reasonable to look for a larger space which could be used instead of $\LU(\Omega)$ to overcome this issue. Therefore we define the \emph{Temam--Strang space} $\UU(\Omega)$ as a subspace of $\BD(\Omega)$:
\[
\UU(\Omega)
:=
\left\{u\in \BD(\Omega): \ \div u\in \LL^2(\Omega) \right\}.
\]
The space $ \UU(\Omega) $ is endowed with the norm
\[
\norm{u}_{\UU}
:=
\norm{u}_{\BD}+\norm{\div u}_{\LL^2},
\]
which turns it into a Banach space. Similarly to the space $\BD$, one usually works in weaker topologies than the norm topology. We distinguish three such topologies in the following.

\begin{definition}[Weak* convergence]\label{def:weak-conv}
	We say that $ (u_h)\subset\UU(\Omega) $ converges \emph{weakly*} to $u$ in $\UU(\Omega)$ if $ u_h\to u $ strongly in $\LL^1(\Omega;\mathbb{R}^d)$, $ Eu_h\starconv Eu $ weakly* in $ \M(\Omega;\mathbb{R}^{d\times d}_{\mathrm{sym}}) $ and $ \div u_h\rightharpoonup \div u $ weakly in $\LL^2(\Omega)$.
\end{definition}

We have the following simple fact.
\begin{lemma}\label{lem:convInU}
	Let $(u_h)\subset\UU(\Omega)$ be a sequence such that $u_h\to u$ strongly in $\LL^1(\Omega;\mathbb{R}^d)$ and $(u_h)$ is uniformly norm-bounded in $\UU(\Omega)$. Then, $(u_h)$ converges weakly* to $u$ in $\UU(\Omega)$.
\end{lemma}

Note that the same result holds for a sequence in $\BD(\Omega)$ and both statements can be proved similarly to the proof of Proposition~3.13 in~\cite{AmbrosioFuscoPallara00}.

\begin{definition}[Strict convergence]
	We say that a sequence $ (u_h)\subset\UU(\Omega) $ converges \emph{strictly} to $u$ in $\UU(\Omega)$ if $ u_h\to u $ strongly in $\LL^1(\Omega;\mathbb{R}^d)$, $ |Eu_h|(\Omega)\to |Eu|(\Omega) $ and $ \div u_h\to \div u $ strongly in $\LL^2(\Omega)$.
\end{definition}

For a measure $\mu\in\M(\mathbb{R}^d;\mathbb{R}^d)$ with the Lebesgue decomposition
\[
\mu
=
\frac{\dif\mu}{\dif\mathcal{L}^d}\mathcal{L}^d+\mu^s
\]
we define a Borel measure $\langle\mu\rangle:\mathcal{B}(\mathbb{R}^d)\to[0,\infty]$ by
\[
\langle\mu\rangle(A)
:=
\int_{A}\sqrt{1+\left|\frac{\dif\mu}{\dif\mathcal{L}^d}\right|^2}\,\dif x+|\mu^s|(A).
\]
\begin{definition}[Area-strict convergence]
	We say that $ (u_h)\subset\UU(\Omega) $ converges \emph{area-strictly} to $u$ in $\UU(\Omega)$ if $ u_h\to u $ strictly, $ \langle Eu_h\rangle(\Omega)\to \langle Eu\rangle(\Omega) $ and $ \langle \dev Eu_h\rangle(\Omega)\to \langle \dev Eu\rangle(\Omega) $.
\end{definition}
The last type of convergence is particularly important, as it allows approximation of functions in $\UU(\Omega)$ by smooth functions (which is not possible in the norm topology), see Remark~\ref{rem:area-strict-continuous-extension}.

For $u\in\UU(\Omega)$ we have that $ \dev E^su=E^su $, since the trace part of $ Eu $, which is equal to $\div u$, is absolutely continuous with respect to the Lebesgue measure $\mathcal{L}^d$.

\subsection{Generalized convexity}\label{subsec:gen-conv}
In this section we recall some information about weaker notions of convexity. These convexity notions are symmetric counterparts of the usual quasiconvexity in the sense of Morrey~\cite{Morrey52} and rank-one convexity.

\begin{definition}\label{def:sym-qc}
	Let $f:\mathbb{R}^{d\times d}_{\mathrm{sym}}\to\mathbb{R}$ be a locally bounded Borel function. We call $f$ \emph{symmetric-quasiconvex}, provided that for all bounded Lipschitz domains $D\subset\mathbb{R}^d$, all test functions $\psi\in\WW^{1,\infty}_0(D;\mathbb{R}^d)$ and all matrices $A\in\mathbb{R}^{d\times d}_{\mathrm{sym}}$ the inequality
	\begin{equation}\label{eq:symQC}
	f(A)
	\leq
	\frac{1}{|D|}\int_D f(A+\mathcal{E}\psi(y))\,\dif y
	\end{equation}
	holds.
\end{definition}
If the function $f$ additionally satisfies an asymptotic growth condition of the form $|f(A)|\leq C(1+|A|^p)$, $p\in[1,\infty)$, then it is sufficient to test the above inequality with $\psi\in\WW^{1,p}_0(D;\mathbb{R}^d)$ instead of Lipschitz functions (the proof is analogous to Lemma~7.1 in~\cite{Rindler18book}, also cf.~Proposition~3.4 in~\cite{FonsecaMuller99}).
\begin{definition}[Symmetric-quasiconvex envelope]\label{def:SQ-envelope}
	Let $f:\mathbb{R}^{d\times d}_{\mathrm{sym}}\to\mathbb{R}$ be a Borel function. Then, the \emph{symmetric-quasiconvex envelope} $SQf:\mathbb{R}^{d\times d}_{\mathrm{sym}}\to\mathbb{R}\cup\{-\infty\}$ is defined as
	\begin{equation}\label{eq:envelope}
	SQf(A)
	:=
	\inf\left\{ \mint{-}_{D}f(A+\mathcal{E}\psi(y))\,\dif y: \ \psi\in\WW^{1,\infty}_0(D;\mathbb{R}^d) \right\}.
	\end{equation}
\end{definition}
\begin{remark}\label{rem:properties-envelope}\mbox{}
	\begin{enumerate}
		\item By the Vitali covering argument one can show that the inequality~\eqref{eq:symQC} and the formula~\eqref{eq:envelope} are independent of the choice of the domain $D$ (cf.~Lemma~5.2(i) in~\cite{Rindler18book}). See also Proposition~5.11 in \cite{Dacorogna08book} for a different proof.
		\item For a non-negative continuous function $f$ with $p$-growth, $1\leq p<\infty$, the symmetric-quasiconvex envelope $SQf$ is symmetric-quasiconvex and also has $p$-growth (see Lemma~7.1 in~\cite{Rindler18book}).
		\item For a function $f$ as in (2), we can equivalently express the symmetric-quasiconvex envelope of $f$ as the greatest symmetric-quasiconvex function, no larger than $f$, i.e.
		\[
		SQf(A)
		=
		\sup\left\{ g(A): \ g \text{ is symmetric-quasiconvex and } g\leq f \right\}.
		\]
	\end{enumerate}
\end{remark}

\begin{definition}\label{def:sym-r1c}
	Let $f:\mathbb{R}^{d\times d}_{\mathrm{sym}}\to\mathbb{R}$ be a locally bounded Borel function. We call $f$ \emph{symmetric rank-one convex} if
	\[
	\mathbb{R}\ni t\mapsto f(A+ta\odot b)
	\]
	is convex for all $A\in\mathbb{R}^{d\times d}_{\mathrm{sym}}$ and all $a,b\in\mathbb{R}^d$.
\end{definition}

By the one-directional oscillations argument, similar to the one in the proof of Proposition~5.3 in \cite{Rindler18book}, one can prove that for a symmetric-quasiconvex function $f:\mathbb{R}^{d\times d}_{\mathrm{sym}}\to\mathbb{R}$ the inequality
\[
f(\theta A+(1-\theta)B)
\leq
\theta f(A)+(1-\theta)f(B)
\]
holds for $A,B\in\mathbb{R}^{d\times d}_{\mathrm{sym}}$ with $B-A=a\odot b$ for some $a,b\in\mathbb{R}^d$ and $\theta\in[0,1]$. This is equivalent to $f$ being symmetric rank-one convex.

The following convexity result for positively 1-homogeneous functions in conjunction with the $\BD$-analogue of Alberti's rank-one theorem (cf.~Theorem~\ref{thm:bd-rank-one}) plays an important role in the study of the singular part of the relaxation $\mathcal{F}_{*}$ of $\mathcal{F}$.
\begin{theorem}[Kirchheim-Kristensen~\cite{KirchheimKristensen16}]\label{thm:kk-convexity}
	Let $\mathcal{C}$ be an open convex cone in a normed finite-dimensional real vector space $\mathcal{V}$, and let $\mathcal{D}$ be a cone of directions in $\mathcal{V}$ such that $\mathcal{D}$ spans $\mathcal{V}$.

	If $f:\mathcal{C}\to\mathbb{R}$ is $\mathcal{D}$-convex (i.e., its restrictions to line segments in $\mathcal{C}$ in directions of $\mathcal{D}$ are convex) and positively 1-homogeneous, then $f$ is convex at each point of $\mathcal{C}\cap\mathcal{D}$.

	More precisely, and in view of homogeneity, for each $x_0\in\mathcal{C}\cap\mathcal{D}$ there exists a linear function $\ell:\mathcal{V}\to\mathbb{R}$ satisfying $\ell(x_0)=f(x_0)$ and $f\geq\ell$ on $\mathcal{C}$.
\end{theorem}

We also record the following simple fact.
\begin{proposition}\label{prop:span-of-SD}
	The set of symmetric and deviatoric matrices $\SD$ is spanned by the set
	\[
	\mathcal{S}
	:=
	\left\{ a\odot b: \ a,b\in\mathbb{R}^d, \ a\cdot b=0 \right\}.
	\]
\end{proposition}
We draw the following important conclusion from Theorem~\ref{thm:kk-convexity} and Proposition~\ref{prop:span-of-SD}.
\begin{corollary}\label{cor:kk-corollary}
	A symmetric rank-one convex and positively 1-homogeneous function $f:\SD\to\mathbb{R}$ is convex at each point of the symmetric rank-one cone $\mathcal{S}$.
\end{corollary}

\subsection{Functionals}\label{subsec:functionals}
The functional
\begin{equation}\label{eq:non-relaxed}
\LU(\Omega)\ni u
\mapsto \mathcal{F}[u,\Omega]
:=
\int_{\Omega}f(x,\mathcal{E}u(x))\,\dif x
\end{equation}
can be extended to the functional
\begin{align}\label{eq:relaxed}
\begin{split}
\UU(\Omega)\ni u
\mapsto
\overline{\mathcal{F}}[u,\Omega]
&:=
\int_{\Omega}f(x,\mathcal{E}u)\,\dif x
+
\int_{\Omega}f_{\dev}^{\#}\left(x,\frac{\dif E^su}{\dif |E^su|}\right)\,\dif|E^su|.
\end{split}
\end{align}
The following theorem was proved by Jesenko and Schmidt~\cite{JesenkoSchmidt17}:
\begin{theorem}\label{thm:jesenko-schmidt}
	Let $f:\Omega\times\mathbb{R}^{d\times d}_{\mathrm{sym}}\to[0,\infty)$ be a continuous function satisfying the following conditions:
	\begin{enumerate}
		\item\label{enum:hencky-growth} there exist constants $0<m\leq M$ such that for all $(x,A)\in\Omega\times\mathbb{R}^{d\times d}_{\mathrm{sym}}$ the growth estimates
		\begin{equation}\label{eq:hencky-growth}
		m((\tr A)^2+|\dev A|)
		\leq
		f(x,A)
		\leq
		M(1+(\tr A)^2+|\dev A|)
		\end{equation}
		hold;
		\item\label{enum:sym-r1-convexity} $f(x,\cdot)$ is symmetric rank-one convex;
		\item\label{enum:hencky-recession} for every fixed $D\in\SD$ the map $x\mapsto f_{\dev}^{\#}(x,D)$ is continuous; here $f_{\dev}^{\#}$ is the recession function of the restriction $f_{\dev}:=f|_{\Omega\times\SD}$ defined by
		\begin{equation}\label{eq:hencky-recession}
		f_{\dev}^{\#}(x,D)
		:=
		\limsup\limits_{\substack{D'\to D \\ s\to\infty}}\frac{f_{\dev}(x,sD')}{s}.
		\end{equation}
	\end{enumerate}
	Then, the functional~\eqref{eq:non-relaxed} extends continuously, with respect to the area-strict convergence in $\UU(\Omega)$, to the functional~\eqref{eq:relaxed}.
\end{theorem}

\begin{remark}\label{rem:area-strict-continuous-extension}
	For $u\in\UU(\Omega)$ there exists a sequence $(v_h)\subset\LU(\Omega)\cap\CC^\infty(\Omega;\mathbb{R}^d)$ such that $v_h\to u$ area-strictly in $\UU(\Omega)$, see Theorem~14.1.4 in~\cite{Attouch14} (the proof is similar to the proof of Lemma~11.1 in~\cite{Rindler18book}, with the strong $\LL^2$-convergence of $(\div v_h)$ being a consequence of the mollification). In virtue of Theorem~\ref{thm:jesenko-schmidt} we have that
	\begin{align*}
	\int_{\Omega}f(x,\mathcal{E}v_h)\,\dif x
	\to
	\int_{\Omega}f(x,\mathcal{E}u)\,\dif x
	+
	\int_{\Omega}f_{\dev}^{\#}\left(x,\frac{\dif E^su}{\dif|E^su|}\right)\,\dif |E^su|.
	\end{align*}
\end{remark}

\begin{remark}\label{rem:propertiesOfRecessionF}\mbox{}
	\begin{enumerate}
		\item\label{enum:homogeneity} The recession function $f_{\dev}^{\#}(x,\cdot)$ is \emph{positively 1-homogeneous},  i.e., for $\alpha\geq 0$ and $(x,D)\in\Omega\times\SD$ it holds that
		\[
		f_{\dev}^{\#}(x,\alpha D)=\alpha f_{\dev}^{\#}(x,D).
		\]
		\item\label{enum:lipschitz} Since the symmetric rank-one cone $\mathcal{S}$ from Proposition~\ref{prop:span-of-SD} spans $\SD$, the function $f_{\dev}(x,\cdot)$ is globally Lipschitz for every $x\in\Omega$ (this is a consequence of $f_{\dev}(x,\cdot)$ being separately convex with linear growth at infinity and Lemma~5.42 in~\cite{AmbrosioFuscoPallara00}).
		\item\label{enum:sr1c} Since $f_{\dev}(x,\cdot)$ is a symmetric rank-one convex function with linear growth at infinity, the recession function $f_{\dev}^{\#}(x,\cdot)$ is also symmetric rank-one convex and by~\eqref{enum:lipschitz} we can write
		\[
		f_{\dev}^{\#}(x,D)
		=
		\limsup\limits_{s\to\infty}\frac{f_{\dev}(x,sD)}{s}.
		\]
		\item\label{enum:kk-corollary} By Corollary~\ref{cor:kk-corollary} the recession function $f_{\dev}^{\#}$ is convex at points of $\mathcal{S}$.
	\end{enumerate}
\end{remark}

\section{Proof of Theorem~\ref{thm:main-result}}\label{sec:main-result}
Our proof is structured as follows. First, in Lemma~\ref{lem:lsc-linear} we prove that the conclusion of Theorem~\ref{thm:main-result} holds for linear weak* limits. This step is essential for the blow-up argument in the proof of the first part of Proposition~\ref{prop:lower-bound}.

We investigate the relaxation $\mathcal{F}_{*}$ of $\mathcal{F}$ defined in~\eqref{eq:relaxU}. In Proposition~\ref{prop:relaxation-ts} we prove that $\mathcal{F}_{*}$ is lower semicontinuous with respect to the weak* convergence in $\UU(\Omega)$ (see Subsection~\ref{subsubsec:space-T-S} for relevant definitions).

Next, we establish that for all $u\in\UU(\Omega)$ the map $V\mapsto\mathcal{F}_{*}[u,V]$ is a restriction to open sets of a finite Radon measure. We then decompose this measure into the absolutely continuous part $\mathcal{F}_{*}^a$ and the singular part $\mathcal{F}_{*}^s$ (with respect to the Lebesgue measure) and prove the following lower bounds:
\begin{equation}\label{eq:first-bound}
\mathcal{F}_{*}^a[u,B]\geq\int_B f(\mathcal{E}u)\,\dif x
\end{equation}
and
\begin{equation}\label{eq:second-bound}
\mathcal{F}_{*}^s[u,B]\geq\int_B f_{\dev}^{\#}\left(\frac{\dif E^su}{\dif|E^su|}\right)\,\dif |E^su|
\end{equation}
for all Borel sets $B\subset\Omega$. For the proof of the \emph{regular bound}~\eqref{eq:first-bound} we use the classical blow-up sequence argument (cf.~Proposition~5.53 in~\cite{AmbrosioFuscoPallara00}), whereas the proof of the \emph{singular bound}~\eqref{eq:second-bound} relies on the Kirchheim-Kristensen convexity result for positively 1-homogeneous functions~\cite{KirchheimKristensen16}.

Finally, together with the upper bound $\mathcal{F}_{*}\leq\overline{\mathcal{F}}$ from Proposition~\ref{prop:upper-bound} we obtain that $ \mathcal{F}_{*}=\overline{\mathcal{F}} $, thus Theorem~\ref{thm:main-result} follows.

In order to prove Theorem~\ref{thm:main-result} we use cut-off arguments (see Lemmas~\ref{lem:lsc-linear} and~\ref{lem:measure-property}). For a given function $u\in\UU(\Omega)$ and some smooth cut-off function $\varphi\in\CC_c^1(\Omega)$, the product $\varphi u$ is in $\BD(\Omega)$, but not necessarily in $\UU(\Omega)$. Indeed, we have
\[
\div (\varphi u)=\nabla\varphi\cdot u+\varphi\div u
\]
and the first term on the right-hand side does not belong to $\LL^2(\Omega)$ in general.

The following result due to Bogovskii (see~\cite{Bogovskii79,Bogovskii80} or section III.3 in~\cite{Galdi11} for the proof) is essential, since it provides a suitable correction term $v$ such that $\varphi u+v\in\UU(\Omega)$.
\begin{theorem}[Bogovskii]\label{thm:bogovskii}
	Let $\Omega\subset\mathbb{R}^d$ be a bounded Lipschitz domain and $1<q<\infty$. There exists a linear operator $\mathcal{B}:\LL^q(\Omega)\to\WW^{1,q}_0(\Omega;\mathbb{R}^d)$ with the following properties:
	\begin{enumerate}
		\item[(i)] for every $g\in\LL^q(\Omega)$ such that $\int_\Omega g\,\dif x=0$ it holds that
		\[
		\div \mathcal{B}g=g \quad \text{in  } \Omega;
		\]
		\item[(ii)] for every $g\in\LL^q(\Omega)$ the estimate
		\[
		\norm{\nabla(\mathcal{B}g)}_{q}\leq C_q\norm{g}_q
		\]
		holds with a translation- and scaling-invariant constant $C_q>0$, depending only on $\Omega$ and $q$;
		\item[(iii)] if $g\in\CC_c^\infty(\Omega)$, then $\mathcal{B}g\in\CC_c^\infty(\Omega;\mathbb{R}^d)$.
	\end{enumerate}
\end{theorem}

We begin with a series of lemmas. The first lemma asserts that the conclusion of Theorem~\ref{thm:main-result} holds for linear limits.
\begin{lemma}\label{lem:lsc-linear}
	Let $A\in\mathbb{R}^{d\times d}_{\mathrm{sym}}$ and let $(u_h)\subset\UU(\Omega)$ be a sequence such that $u_h\starconv Ax$ weakly* in $\UU(\Omega)$. Then
	\begin{equation}\label{eq:lsc-linear}
	|\Omega|f(A)\leq\liminf_{h\to\infty}\int_{\Omega}f(\mathcal{E}u_h)\,\dif x.
	\end{equation}
\end{lemma}
\begin{proof}
	In view of Theorem~\ref{thm:jesenko-schmidt} and Remark~\ref{rem:area-strict-continuous-extension} we can without loss of generality assume that $(u_h)\subset\LU(\Omega)\cap\CC^\infty(\Omega;\mathbb{R}^d)$. The proof is divided into two steps. In the first step we prove~\eqref{eq:lsc-linear} for a sequence $(u_h)$ which has linear boundary values. Then, in the second step we prove, using a cut-off argument, that the assumption of the linear boundary values can be dropped.

	\textit{Step 1.} Suppose that $u_h(x)-Ax$ is compactly supported inside $\Omega$ for all $h\in\mathbb{N}$ and take $\psi_h(x):=u_h(x)-Ax$. Clearly, $\psi_h\in\WW^{1,\infty}_0(\Omega;\mathbb{R}^d)$. Then, by the symmetric-quasiconvexity of $f$ we obtain
	\begin{align*}
	|\Omega|f(A)
	\leq
	\int_{\Omega}f(A+\mathcal{E}\psi_h(y))\,\dif y
	=
	\int_{\Omega}f(\mathcal{E}u_h(y))\,\dif y
	\end{align*}
	for all $h\in\mathbb{N}$. Therefore,
	\[
	|\Omega|f(A)
	\leq
	\liminf_{h\to\infty}\int_{\Omega}f(\mathcal{E}u_h)\,\dif x.
	\]

	\textit{Step 2.} Let $u_h\starconv Ax$ weakly* in $\UU(\Omega)$. Fix $n\in\mathbb{N}$ and $\varepsilon>0$ and choose a Lipschitz subdomain $\Omega_0\Subset\Omega$ such that $|\Omega\setminus\Omega_{0}|\leq\varepsilon$. Let $R:=\dist(\Omega_0,\partial\Omega)$ and for $i=1,\ldots,n$ define the sets
	\[
	\Omega_i
	:=
	\left\{x\in\Omega: \ \dist(x,\Omega_0)<\frac{iR}{n} \right\}.
	\]
	Now, choose cut-off functions $\varphi_i\in\CC_c^1(\Omega;[0,1])$ such that
	\begin{equation}\label{eq:cut-off}
	\mathbbm{1}_{\Omega_{i-1}}\leq\varphi_i\leq\mathbbm{1}_{\Omega_{i}} \quad \text{and} \quad
	\norm{\nabla\varphi_i}_\infty\leq\frac{2n}{R}
	\end{equation}
	and for $x\in\Omega$ define
	\[
	u_{h,i}(x):=Ax+\varphi_i(x)(u_h(x)-Ax).
	\]
	We have
	\begin{equation}\label{eq:Step1-linear}
	\mathcal{E}u_{h,i}=A+\varphi_i(\mathcal{E}u_h-A)+\nabla\varphi_i\odot (u_h-Ax)
	\end{equation}
	and
	\begin{equation}\label{eq:Step1-linear2}
	\div u_{h,i}
	=
	\tr A+\varphi_i(\div u_h-\tr A)+\nabla\varphi_i\cdot (u_h-Ax).
	\end{equation}
	Note that the last term in~\eqref{eq:Step1-linear2} belongs only to $\LL^{d/(d-1)}(\Omega)$ by the embedding $\BD(\Omega)\subset\LL^q(\Omega;\mathbb{R}^d)$ for $1\leq q\leq d/(d-1)$, thus $u_{h,i}\not\in\UU(\Omega)$ for $d>2$. In order to overcome this problem we fix some $1<q<d/(d-1)$ and define numbers
	\[
	\xi_{h,i}
	:=
	\frac{1}{|S_i|}\int_{S_i}\nabla\varphi_i(x)\cdot (u_h(x)-Ax)\,\dif x,
	\]
	where $S_i:=\Omega_i\setminus\overline{\Omega_{i-1}}$ is the open strip between $\Omega_{i-1}$ and $\Omega_i$. Note that $\supp\nabla\varphi_i\subset S_i$. Define
	\begin{equation}\label{eq:Step1-functions}
	f_{h,i}
	:=
	-\nabla\varphi_i\cdot (u_h-Ax)+\xi_{h,i}\in\LL^q(S_i).
	\end{equation}
	By Theorem~\ref{thm:bogovskii} there exist functions $z_{h,i}\in\WW^{1,q}_0(S_i;\mathbb{R}^d)$ such that
	\[
	\div z_{h,i}=f_{h,i} \quad \text{in } S_i
	\]
	and such that the estimate
	\begin{equation}\label{eq:Step2-estimate}
	\norm{\nabla z_{h,i}}_q\leq C_q\norm{f_{h,i}}_q
	\end{equation}
	holds. We also extend the functions $z_{h,i}$ by zero outside $S_i$. Let $w_{h,i}\in\UU(\Omega)$ be defined as
	\[
	w_{h,i}:=u_{h,i}+z_{h,i}.
	\]
	The correction term $z_{h,i}$ ensures that $\div w_{h,i}\in\LL^2(\Omega)$.
	
	Henceforth, for simplicity we write $C>0$ for a generic constant that changes from line to line, possibly depending on $\Omega, M, A, R, n, q$, but never on $h,i$. Note that we have the following estimate:
	\begin{align}\label{eq:Step2-estimate2}
	\norm{f_{h,i}}_q
	\leq
	C\norm{u_h-Ax}_q.
	\end{align}
	This estimate, in conjunction with the Poincar{\'e} inequality,~\eqref{eq:Step2-estimate}, and the compactness of the embedding $\BD(\Omega)\Subset\LL^q(\Omega;\mathbb{R}^d)$, implies that $z_{h,i}\to 0$ in $\WW^{1,q}(\Omega;\mathbb{R}^d)$ as $h\to\infty$. 
	
	Note that the sequence $(w_{h,i})_h$ is bounded in $\UU(\Omega)$ for fixed $i$. Indeed, this is a consequence of the weak* convergence $u_h\starconv Ax$ in $\UU(\Omega)$ and the estimate
	\[
	\xi_{h,i}^2
	\leq C\biggl(\min_{\ell\in\{1,\ldots,n\}}|S_\ell|\biggr)^{-2/q}\norm{u_h-Ax}_q^2 
	\leq C \norm{u_h-Ax}_q^2,
	\]
	since $|S_i|>0$ for all $i\in\{1,\ldots,n\}$.

	Since $w_{h,i}\to Ax$ in $\LL^1(\Omega;\mathbb{R}^d)$ as $h \to \infty$, and $(w_{h,i})_h$ is bounded in $\UU(\Omega)$ for all $i=1,\ldots,n$, by Lemma~\ref{lem:convInU} it follows that $w_{h,i}\starconv Ax$ weakly* in $\UU(\Omega)$. Moreover, $w_{h,i}|_{\partial\Omega}=Ax$ for every $i=1,\ldots,n$ and $h\in\mathbb{N}$.

	By the upper growth bound~\eqref{eq:growth} we obtain
	\begin{align*}
	\int_{\Omega}f(\mathcal{E}w_{h,i})\,\dif x
	&=
	\int_{\overline{\Omega_{i-1}}}f(\mathcal{E}u_h)\,\dif x
	+
	\int_{S_i}f(\mathcal{E}w_{h,i})\,\dif x
	+
	\int_{\Omega\setminus\Omega_i}f(A)\,\dif x \\
	&\leq
	\int_{\Omega}f(\mathcal{E}u_h)\,\dif x
	+
	\int_{S_i}f(\mathcal{E}w_{h,i})\,\dif x
	+
	|\Omega\setminus\Omega_0|f(A) \\
	&\leq
	\int_{\Omega}f(\mathcal{E}u_h)\,\dif x
	+
	M\int_{S_i} |\dev\mathcal{E}w_{h,i}|+|\div w_{h,i}|^2\,\dif x \\
	&\quad+
	C|\Omega\setminus\Omega_0|.
	\end{align*}
	The estimates~\eqref{eq:Step2-estimate} and~\eqref{eq:Step2-estimate2} together with H{\"o}lder's inequality yield
	\begin{align*}
	\int_{S_i}|\dev\mathcal{E}z_{h,i}|\,\dif x
	\leq
	C|\Omega\setminus\Omega_{0}|^{1/q'}\,\sup_h\norm{u_h-Ax}_q,
	\end{align*}
	where $1/q+1/q'=1$.
	Therefore, since $|\Omega\setminus\Omega_{0}|\leq\varepsilon$, we estimate
	\begin{align*}
	\int_{S_i}|\dev \mathcal{E}w_{h,i}|\,\dif x
	&\leq
	|\dev A|\,|S_i|
	+
	\int_{S_i}|\varphi_i|\,|\dev\mathcal{E}u_h-\dev A|\,\dif x \\
	&\quad+
	\int_{S_i}|\dev[\nabla\varphi_i\odot (u_h-Ax)]|\,\dif x
	+
	\int_{S_i}|\dev\mathcal{E}z_{h,i}|\,\dif x \\
	&\leq
	|\dev A|\,|\Omega\setminus\Omega_{0}|+\int_{S_i}|\dev\mathcal{E}u_h-\dev A|\,\dif x \\
	&\quad+
	\int_{S_i}|\dev[\nabla\varphi_i\odot (u_h-Ax)]|\,\dif x
	+
	C|\Omega\setminus\Omega_{0}|^{1/q'} \\
	&\leq
	C(\varepsilon+\varepsilon^{1/q'})+\int_{S_i}|\dev\mathcal{E}u_h-\dev A|\,\dif x \\
	&\quad+
	\frac{4n}{R}\int_{\Omega}|u_h(x)-Ax|\mathbbm{1}_{S_i}(x)\,\dif x \\
	&\leq
	C(\varepsilon+\varepsilon^{1/q'})+\int_{S_i}|\dev\mathcal{E}u_h-\dev A|\,\dif x \\
	&\quad+
	\frac{4n}{R}\sup_h\norm{u_h-Ax}_q\,|S_i|^{1/q'} \\
	&\leq
	C(\varepsilon+\varepsilon^{1/q'})+\int_{S_i}|\dev\mathcal{E}u_h-\dev A|\,\dif x.
	\end{align*}
	Next, we estimate the divergence term:
	\begin{align*}
	&\int_{S_i}|\div w_{h,i}|^2\,\dif x \\
	&\quad\leq
	\int_{S_i}|\tr A+\varphi_i(\div u_h-\tr A)+\xi_{h,i}|^2\,\dif x \\
	&\quad\leq
	3\int_{S_i}|\tr A|^2+|\div u_h-\tr A|^2+\xi_{h,i}^2\,\dif x \\
	&\quad\leq
	3|\tr A|^2|\Omega\setminus\Omega_{0}|+3\int_{S_i}|\div u_h-\tr A|^2\,\dif x
	+
	\frac{12n^2}{R^2|S_i|}\norm{u_h-Ax}_1^2 \\
	&\quad\leq
	C\varepsilon+3\int_{S_i}|\div u_h-\tr A|^2\,\dif x
	+
	\frac{12n^2}{R^2|S_i|}\norm{u_h-Ax}_1^2,
	\end{align*}
	where we used the inequality
	\begin{align*}
	\xi_{h,i}^2
	=
	\frac{1}{|S_i|^2}\left(\int_{S_i}\nabla\varphi_i(x)\cdot (u_h(x)-Ax)\,\dif x\right)^2
	\leq
	\frac{4n^2}{R^2|S_i|^2}\norm{u_h-Ax}_1^2.
	\end{align*}
	Combining the above estimates yields
	\begin{align*}
	\int_{\Omega}f(\mathcal{E}w_{h,i})\,\dif x
	&\leq
	\int_{\Omega}f(\mathcal{E}u_h)\,\dif x
	+
	M\int_{S_i}|\dev\mathcal{E}u_h-\dev A|\,\dif x \\
	&\quad+
	3M\int_{S_i}|\div u_h-\tr A|^2\,\dif x
	+
	C(\varepsilon+\varepsilon^{1/q'}) \\
	&\quad+
	\frac{12Mn^2}{R^2|S_i|}\norm{u_h-Ax}_1^2.
	\end{align*}
	By Step 1 we have
	\begin{align*}
	|\Omega|f(A)
	&\leq
	\liminf_{h\to\infty}\int_{\Omega}f(\mathcal{E}w_{h,i})\,\dif x \\
	&\leq
	\liminf_{h\to\infty}\bigg[\int_{\Omega}f(\mathcal{E}u_h)\,\dif x+M\int_{S_i}|\dev\mathcal{E}u_h-\dev A|\,\dif x \\
	&\qquad\qquad\quad+
	3M\int_{S_i}|\div u_h-\tr A|^2\,\dif x
	+
	\frac{12n^2}{R^2|S_i|}\norm{u_h-Ax}_1^2 \bigg] \\
	&\qquad\qquad\quad+
	C(\varepsilon+\varepsilon^{1/q'}).
	\end{align*}
	Since $u_h\to Ax$ strongly in $\LL^1(\Omega;\mathbb{R}^d)$, the term
	\[
	\frac{12n^2}{R^2|S_i|}\norm{u_h-Ax}_1^2
	\]
	vanishes as $h\to\infty$. Summing up over $i=1,\ldots,n$, dividing by $n$, and using the superadditivity of a~lower limit yields
	\begin{align*}
	|\Omega|f(A)
	&\leq
	\liminf_{h\to\infty}\int_{\Omega}f(\mathcal{E}w_{h,i})\,\dif x \\
	&\leq
	\liminf_{h\to\infty}\int_{\Omega}f(\mathcal{E}u_h)\,\dif x+\frac{M}{n}\sup_h\int_{\Omega}|\dev\mathcal{E}u_h-\dev A|\,\dif x \\
	&\quad+
	\frac{3M}{n}\sup_h\int_{\Omega}|\div u_h-\tr A|^2\,\dif x
	+
	C(\varepsilon+\varepsilon^{1/q'}).
	\end{align*}
	Letting $\varepsilon\downarrow 0$ and $n\to\infty$ yields
	\begin{equation*}
	|\Omega|f(A)
	\leq
	\liminf_{h\to\infty}\int_{\Omega}f(\mathcal{E}u_h)\,\dif x. \qedhere
	\end{equation*}
\end{proof}

\begin{remark}
	Clearly, Lemma~\ref{lem:lsc-linear} also holds for affine limits.
\end{remark}

We are now going to prove that the relaxation
\[
\mathcal{F}_{*}[u,\Omega]:=\inf\left\{\liminf_{h\to\infty} \mathcal{F}[u_h,\Omega]: \ (u_h)\subset\LU(\Omega), \ u_h\starconv u \text{ in } \UU(\Omega)\right\}
\]
satisfies the lower bound
\begin{equation}\label{eq:lower-bound}
\mathcal{F}_{*}[u,\Omega]
\geq
\int_{\Omega}f(\mathcal{E}u)\,\dif x
+
\int_{\Omega}f_{\dev}^{\#}\left(\frac{\dif E^su}{\dif|E^su|}\right)\,\dif|E^su|.
\end{equation}

We first prove that the relaxation is weakly* lower semicontinuous on $ \UU(\Omega) $, for which we need the following lemma~(for a proof see Lemma~11.1.1 in~\cite{Attouch14}).

\begin{lemma}[Diagonalization lemma]\label{lem:diagonalisation}
	Let $(a_{k,l})_{k,l}\subset X$ be a doubly-indexed sequence in a first-countable topological space $X$ such that
	\begin{enumerate}
		\item $\lim_{l\to\infty}a_{k,l}=a_k$,
		\item $\lim_{k\to\infty}a_k=a$.
	\end{enumerate}
	Then, there exists a non-decreasing map $l\mapsto k(l)$ such that
	\[
	\lim_{l\to\infty}a_{k(l),l}=a.
	\]
\end{lemma}
We apply Lemma~\ref{lem:diagonalisation} in Proposition~\ref{prop:relaxation-ts} below with $X=\mathcal{B}$, where $\mathcal{B}\subset\UU(\Omega)$ is a norm-bounded set. This way, $X$ endowed with the weak* topology of $\UU(\Omega)$ is metrizable, thus first-countable.

\begin{proposition}\label{prop:relaxation-ts}
	The relaxation $\mathcal{F}_{*}$ is lower semicontinuous with respect to weak* convergence in $\UU(\Omega)$.
\end{proposition}
\begin{proof}
	We argue by contradiction. To this end suppose there is a sequence $(u_j)\subset\UU(\Omega)$ such that $u_j\starconv u$ for some $u\in\UU(\Omega)$ and
	\[
		\mathcal{F}_{*}[u,\Omega]
		>
		\liminf_{h\to\infty} \mathcal{F}_{*}[u_j,\Omega].
	\]
	Let $(u_k)_k:=(u_{j_k})_k$ be a subsequence of $ (u_j)_j $ such that
	\[
	\lim_{k\to\infty}\mathcal{F}_{*}[u_k,\Omega]=\liminf_{j\to\infty}\mathcal{F}_{*}[u_j,\Omega].
	\]
	For each $k \in \mathbb{N}$, one can find a sequence $(v_k^{(l)})_l\subset\LU(\Omega)$ such that $v_k^{(l)}\starconv u_k$ and
	\[
	\liminf_{j\to\infty}\mathcal{F}_{*}[u_j,\Omega]
	=
	\lim_{k\to\infty}\mathcal{F}_{*}[u_k,\Omega]
	=
	\lim_{k\to\infty}\lim_{l\to\infty}\mathcal{F}[v_k^{(l)},\Omega].
	\]
	By the lower bound on $\mathcal{F}$ we obtain
	\begin{align*}
	\mathcal{F}_{*}[u,\Omega]
	&>
	\lim_{k\to\infty}\lim_{l\to\infty}\mathcal{F}[v_k^{(l)},\Omega] \\
	&\geq
	\limsup_{k\to\infty}\limsup_{l\to\infty}m\left(\norm{\div v_k^{(l)}}_{\LL^2}^2+\norm{\dev\mathcal{E}v_k^{(l)}}_{\LL^1}\right)
	\end{align*}
	Therefore, the sequence $(v_k^{(l)})$ is uniformly (with respect to both $k$ and $l$) norm-bounded in $\UU(\Omega)$, so we can find a large enough ball $\mathcal{B}\subset\UU(\Omega)$ and apply Lemma~\ref{lem:diagonalisation} to the doubly-indexed sequence
	\[
	(v_k^{(l)},\mathcal{F}[v_k^{(l)},\Omega])_{k,l}\subset \mathcal{B}\times(\mathbb{R}\cup\{+\infty\}).
	\]
	Hence, there exists a sequence $(k_l)_l$ such that $v_{k_l}^{(l)}\starconv u$ as $l\to\infty$ and
	\[
	\lim_{l\to\infty} \mathcal{F}[v_{k_l}^{(l)},\Omega]=\liminf_{j\to\infty}\mathcal{F}_{*}[u_j,\Omega].
	\]
	We have
	\begin{align*}
	\mathcal{F}_{*}[u,\Omega]
	&>
	\lim_{l\to\infty} \mathcal{F}[v_{k_l}^{(l)},\Omega] \\
	&\geq
	\inf\left\{\liminf_{h\to\infty} \mathcal{F}[z_h,\Omega]: \ (z_h)\subset\LU(\Omega), \ z_h\starconv u \text{ in} \UU(\Omega) \right\} \\
	&=
	\mathcal{F}_{*}[u,\Omega],
	\end{align*}
	which is absurd. Therefore, the relaxation $\mathcal{F}_{*}$ is weakly* lower semicontinuous in $\UU(\Omega)$. \qedhere
\end{proof}

\begin{remark}
	Note that the relaxation $\mathcal{F}_{*}$ can be written as
	\[
	\mathcal{F}_{*}[u,\Omega]=\inf\left\{\liminf_{h\to\infty} \mathcal{F}[u_h,\Omega]: \ (u_h)\subset\LU(\Omega), \ u_h\to u \text{ in } \LL^1(\Omega;\mathbb{R}^d)\right\}.
	\]
	Indeed, if this was false, we could find a sequence $ (u_h)\subset\LU(\Omega) $ with $ u_h\to u $ strongly in $\LL^1(\Omega;\mathbb{R}^d)$ such that
	\[
	\mathcal{F}_{*}[u,\Omega]
	>
	\lim_{h\to\infty}\mathcal{F}[u_h,\Omega]
	\geq
	\limsup_{h\to\infty}m\left(\norm{\div u_h}_{\LL^2}^2+\norm{\dev\mathcal{E}u_h}_{\LL^1}\right),
	\]
	where the last inequality follows from the lower bound on the integrand $f$. We see that $(u_h)$ is uniformly norm-bounded in $ \UU(\Omega) $, hence $ u_h\starconv u $ weakly* in $\UU(\Omega)$ by Lemma~\ref{lem:convInU}, whereby we get the contradiction $\mathcal{F}_{*}[u,\Omega]>\mathcal{F}_{*}[u,\Omega]$.
\end{remark}

\begin{remark}\label{rem:invariances}
	The functional $\mathcal{F}_{*}$ satisfies the following properties.
	\begin{enumerate}
		\item[(1)]\label{enum:rigid-invariance} For any rigid deformation $R:\mathbb{R}^d\to\mathbb{R}^d$, i.e., $R(x)=Wx+b $ for $x\in\mathbb{R}^d$, where $ W\in\mathbb{R}^{d\times d}_{\mathrm{skew}} $ is a skew-symmetric matrix and $ b\in\mathbb{R}^d $ is a vector, we have the \textbf{rigid invariance}
		\[
		\mathcal{F}_{*}[u+R,\Omega]
		=
		\mathcal{F}_{*}[u,\Omega].
		\]
		\item[(2)]\label{enum:translation-invariance} For any $x_0\in\mathbb{R}^d$ we have the \textbf{translation invariance}
		\[
		\mathcal{F}_{*}[u(\cdot-x_0),x_0+\Omega]
		=
		\mathcal{F}_{*}[u,\Omega].
		\]
		\item[(3)]\label{enum:blow-up} Let $(R_r)_{r>0}:\mathbb{R}^d\to\mathbb{R}^d$ be a family of rigid deformations. Then, for a \textbf{blow-up} of the form
		\[
		u_r(y)=\frac{u(x_0+ry)-u(x_0)}{r}+R_r(y)
		\]
		where $r>0$ and $y\in(\Omega-x_0)/r$, we have the scaling property
		\[
		\mathcal{F}_{*}\left[u_r,\frac{\Omega-x_0}{r}\right]
		=
		r^{-d}\mathcal{F}_{*}[u,\Omega].
		\]
	\end{enumerate}
\end{remark}

In order to prove the lower bound, we appeal to Lemma~\ref{lem:measure-property} below, which asserts that for a given $u\in\UU(\Omega)$ the map $V\mapsto\mathcal{F}_{*}[u,V]$ is the restriction to the open subsets of $\Omega$ of a Radon measure on $\Omega$, which we still denote by $\mathcal{F}_{*}[u,\cdot]$. Then, we decompose this measure into the absolutely continuous and singular parts with respect to the Lebesgue measure, i.e.
\[
\mathcal{F}_{*}[u,\cdot]=\mathcal{F}_{*}^a[u,\cdot]+\mathcal{F}_{*}^s[u,\cdot], \quad \mathcal{F}_{*}^a[u,\cdot] \ll \mathcal{L}^d\mres\Omega, \quad \mathcal{F}_{*}^s[u,\cdot]\perp\mathcal{L}^d\mres\Omega
\]
and then prove that
\[
\mathcal{F}_{*}^a[u,B]\geq\int_B f(\mathcal{E}u)\,\dif x
\quad\text{and}\quad
\mathcal{F}_{*}^s[u,B]\geq\int_{B}f_{\dev}^{\#}\left(\frac{\dif E^su}{\dif|E^su|}\right)\,\dif|E^su|
\]
for any Borel set $B\subset\Omega$.

\begin{lemma}\label{lem:measure-property}
	For all $u\in\UU(\Omega)$ the set function $V\mapsto\mathcal{F}_{*}[u,V]$ ($V\subset\Omega$ open) is the restriction to the open subsets of $\Omega$ of a finite Radon measure.
\end{lemma}
\begin{proof}
	Fix $u\in\UU(\Omega)$.

	\textit{Step 1.} Let $A', A'', B$ be open subsets of $\Omega$ such that $A'\Subset A''$. We first prove that
	\begin{equation}\label{eq:inequality-measure-property}
	\mathcal{F}_{*}[u,A'\cup B]\leq\mathcal{F}_{*}[u,A'']+\mathcal{F}_{*}[u,B].
	\end{equation}
	Fix $\varepsilon>0$. By the definition of relaxation we can find sequences $(u_h^\varepsilon)\subset\LU(A'')$ and $(v_h^\varepsilon)\subset\LU(B)$ such that $u_h^\varepsilon\starconv u$ weakly* in $\UU(A'')$, $v_h^\varepsilon\starconv u$ weakly* in $\UU(B)$,
	\begin{equation}\label{eq:one}
		\mathcal{F}[u_h^{\varepsilon},A'']\leq\mathcal{F}_{*}[u, A'']+\varepsilon,
	\end{equation}
	and
	\begin{equation}\label{eq:two}
		\mathcal{F}[v_h^{\varepsilon},B]\leq\mathcal{F}_{*}[u, B]+\varepsilon.
	\end{equation}
	Henceforth, we omit the dependence of sequences $u_h$ and $v_h$ on $\varepsilon$. For each $h\in\mathbb{N}$ extend the functions $u_h$ and $v_h$ by zero outside $A''$ and $B$, respectively. Let
	\begin{equation} \label{eq:C_eps}
	C_\varepsilon := \sup_{h \in \mathbb{N}} \, \biggl( \int_{A''}1+|\div u_h|^2+|\mathcal{E}u_h|\,\dif x
	+
	\int_{B}1+|\div v_h|^2+|\mathcal{E}v_h|\,\dif x
	\biggr) < \infty.
	\end{equation}
	Fix $k\in\mathbb{N}$ and an increasing family of open sets
	\[
	A'=A_0\Subset A_1\Subset\ldots\Subset A_k\Subset A''.
	\]
	For each $i=1,\ldots,k$ choose the cut-off function $\varphi_i\in\CC_c^1(A_i;[0,1])$ such that $\varphi_i\equiv 1$ on $A_{i-1}$. Next, define maps $\tilde{w}_{h,i} \in \LL^1(A'\cup B;\mathbb{R}^d)$ via
	\[
	\tilde{w}_{h,i}
	:=
	\varphi_i u_h+(1-\varphi_i)v_h, \quad h\in\mathbb{N}, \quad i=1,\ldots,k.
	\]
	It is clear that $\tilde{w}_{h,i}\in\LU(A_{i-1})$, but $\tilde{w}_{h,i}\not\in\LU(A'\cup B)$, since
	\[
	\div\tilde{w}_{h,i}
	=
	\varphi_i\div u_h+(1-\varphi_i)\div v_h
	+
	\nabla\varphi_i\cdot(u_h-v_h)
	\]
	and the last term on the right-hand side belongs only to $\LL^{d/(d-1)}(A'\cup B)$. To overcome this problem, as before we fix some $1<q<d/(d-1)$ and define
	\[
	\xi_{h,i}
	:=
	\frac{1}{|S_i|}\int_{S_i}\nabla\varphi_i(x)\cdot(u_h(x)-v_h(x))\,\dif x,
	\]
	where $S_i:=A_{i}\setminus\overline{A_{i-1}}$ for $i=1,\ldots,k$. Note that $\supp\nabla\varphi_i\Subset S_i$. By Theorem~\ref{thm:bogovskii} applied in $S_i$ and with the right-hand side
	\[
	f_{h,i}
	:=
	-\nabla\varphi_i\cdot(u_h-v_h)
	+
	\xi_{h,i}\in\LL^q(S_i),
	\]
	there exist functions $z_{h,i}:=\mathcal{B}f_{h,i}\in\WW^{1,q}_0(S_i;\mathbb{R}^d)$ such that
	\[
	\div z_{h,i}=f_{h,i} \quad \text{on } S_i
	\]
	and the estimate
	\begin{equation}\label{eq:estimate}
	\norm{\nabla z_{h,i}}_{q}
	\leq
	C\norm{f_{h,i}}_{q}
	\end{equation}
	holds. We also extend $z_{h,i}$ by zero outside $S_i$.
	Define
	\[
	w_{h,i}:=\tilde{w}_{h,i}+z_{h,i}.
	\]
	The correction term $z_{h,i}$ guarantees that $w_{h,i}\in\LU(A'\cup B)$. Indeed,
	\[
	\div w_{h,i}
	=
	\varphi_i\div u_h+(1-\varphi_i)\div v_h+\xi_{h,i}\mathbbm{1}_{S_i},
	\]
	which clearly belongs to $\LL^2(A'\cup B)$. We have
	\begin{align*}
	\mathcal{F}[w_{h,i},A'\cup B]
	&=
	\int_{A'\cup B}f(\mathcal{E}w_{h,i})\,\dif x \\
	&=
	\int_{(A'\cup B)\cap \overline{A_{i-1}}}f(\mathcal{E}u_h)\,\dif x
	+
	\int_{B\setminus A_i}f(\mathcal{E}v_h)\,\dif x
	+
	\int_{B\cap S_i}f(\mathcal{E}w_{h,i})\,\dif x,
	\end{align*}
	where we used the fact that the corrector $z_{h,i}$ vanishes outside of $S_i$. Hence,
	\begin{align*}
	\mathcal{F}[w_{h,i},A'\cup B]
	&\leq
	\mathcal{F}[u_h,A'']+\mathcal{F}[v_h,B]
	+
	\int_{B\cap S_i}f(\mathcal{E}w_{h,i})\,\dif x.
	\end{align*}
	The last integral can be estimated as follows:
	\begin{alignat*}{3}
	\int_{B\cap S_i}f(\mathcal{E}w_{h,i})\,\dif x
	&\leq
	M\int_{B\cap S_i} && 1+|\div w_{h,i}|^2+|\mathcal{E}w_{h,i}|\,\dif x \\
	&\leq
	3 M\int_{B\cap S_i} && 1+|\div u_h|^2+|\div v_h|^2+\xi_{h,i}^2 \\
	& &&+
	C_k|u_h-v_h|+|\mathcal{E}u_h|+|\mathcal{E}v_h|+|\mathcal{E}z_{h,i}|\,\dif x,
	\end{alignat*}
	where $C_k:=\sup\,\{\norm{\nabla\varphi_i}_\infty: \ 1\leq i\leq k \}$. We have for $1\leq i\leq k$ that
	\begin{align*}
	\xi_{h,i}^2
	&=
	\frac{1}{|S_i|^2}\left(\int_{S_i}\nabla\varphi_i(x)\cdot(u_h(x)-v_h(x))\,\dif x\right)^2 \\
	&\leq
	\frac{C_k^2}{|S_i|^2}\norm{u_h-v_h}_1^2 \\
	&\leq
	C_k^2|S_i|^{-2/q}\norm{u_h-v_h}_q^2.
	\end{align*}
	Here and in all of the following the norms are with respect to the domain $A'\cup B$.
	Since $|S_i|>0$ for all $i\in\{1,\ldots,k\}$, we get
	\[
	\xi_{h,i}^2\leq C_k^2\biggl(\min_{\ell\in\{1,\ldots,k\}}|S_\ell|\biggr)^{-2/q}\norm{u_h-v_h}_q^2 \leq \tilde{C}_{\Omega,q,k} \norm{u_h-v_h}_q^2.
	\]
	By the estimate~\eqref{eq:estimate} and H{\"o}lder's inequality we obtain similarly
	\begin{align*}
	\int_{B\cap S_i}|\mathcal{E}z_{h,i}|\,\dif x
	\leq
	\norm{\nabla z_{h,i}}_q\,|B\cap S_i|^{1/q'}
	\leq
	\tilde{C}_{\Omega,q,k}\norm{u_h-v_h}_q,
	\end{align*}
	where $1/q+1/q'=1$.	Note that for every $h\in\mathbb{N}$ there exists $i_h\in\{1,\ldots,k\}$ such that
	\begin{align*}
	&\int_{B\cap S_{i_h}} 1+|\div u_h|^2+|\div v_h|^2+|\mathcal{E}u_h|+|\mathcal{E}v_h|\,\dif x \\
	&\quad\leq
	\frac{1}{k}\int_{B\cap(\overline{A_k}\setminus A_0)} 1+|\div u_h|^2+|\div v_h|^2+|\mathcal{E}u_h|+|\mathcal{E}v_h|\,\dif x \\
	&\quad\leq
	\frac{C_\varepsilon}{k},
	\end{align*}
	where $C_\varepsilon$ is defined in~\eqref{eq:C_eps}.
	Therefore, combining the above estimates yields
	\[
	\int_{B\cap S_{i_h}}f(\mathcal{E}w_{h,i_h})\,\dif x
	\leq
	C_{\Omega,M,q,k}\bigg(\norm{u_h-v_h}_q^2+\norm{u_h-v_h}_1 + \norm{u_h-v_h}_q\bigg) + \frac{3MC_\varepsilon}{k}.
	\]
	Hence, since $(u_h)$ and $(v_h)$ are chosen such that $\eqref{eq:one}$ and $\eqref{eq:two}$ hold, we have
	\begin{align*}
	\mathcal{F}[w_{h,i_h},A'\cup B]
	&\leq
	\mathcal{F}[u_h,A'']
	+
	\mathcal{F}[v_h,B] \\
	&\quad+
	C_{\Omega,M,q,k}\bigg(\norm{u_h-v_h}_q^2+\norm{u_h-v_h}_1 + \norm{u_h-v_h}_q\bigg) \\
	&\quad + \frac{3MC_\varepsilon}{k} \\
	&\leq
	\mathcal{F}_{*}[u,A'']
	+
	\mathcal{F}_{*}[u,B]
	+2\varepsilon \\
	&\quad+
	C_{\Omega,M,q,k}\bigg(\norm{u_h-v_h}_q^2+\norm{u_h-v_h}_1 + \norm{u_h-v_h}_q\bigg) \\
	&\quad + \frac{3MC_\varepsilon}{k}.
	\end{align*}

	Note that $w_{h,i_h}\to u$ strongly in $\LL^1(A'\cup B;\mathbb{R}^d)$ and $(w_{h,i_h})_h$ is uniformly norm-bounded in $\UU(A'\cup B)$. Lemma~\ref{lem:convInU} thus implies that $(w_{h,i_h})_h$ converges weakly* to $u$ in $\UU(A'\cup B)$. Moreover, $(u_h-v_h)_h$ converges strongly to zero in $\LL^q(A'\cup B;\mathbb{R}^d)$. Therefore we obtain
	\begin{align*}
	\mathcal{F}_{*}[u,A'\cup B]
	&\leq
	\liminf_{h\to\infty}\mathcal{F}[w_{h,i_h},A'\cup B] \\
	&\leq
	\mathcal{F}_{*}[u,A'']+\mathcal{F}_{*}[u,B]+\frac{3MC_\varepsilon}{k} + 2\varepsilon.
	\end{align*}
	Letting $k\to\infty$ followed by $\varepsilon\downarrow 0$ yields the inequality~\eqref{eq:inequality-measure-property}.

	\textit{Step 2.} We now prove that for any open subset $A\subset\Omega$ it holds that
	\begin{equation}\label{eq:inequality-measure-property2}
	\mathcal{F}_{*}[u,A]=\sup\left\{\mathcal{F}_{*}[u,A']: \ A'\Subset A, \ A' \text{ open} \right\}.
	\end{equation}
	In virtue of Remark~\ref{rem:area-strict-continuous-extension} and the growth assumption on the integrand $f$, we obtain the inequality
	\begin{equation}\label{eq:est2}
	\mathcal{F}_{*}[u,A]\leq M\left(\int_A|\div u|^2\,\dif x+\mathcal{L}^d(A)+|Eu|(A)\right).
	\end{equation}

	Therefore, for a fixed $\varepsilon>0$ we can choose a compact set $K\subset A$ such that $\mathcal{F}_{*}[u,A\setminus K]\leq\varepsilon$. Choose open sets $A'$ and $A''$ such that $K\subset A'\Subset A''\Subset A$. By Step~1 with $B=A\setminus K$ we have
	\[
	\mathcal{F}_{*}[u,A]
	\leq
	\mathcal{F}_{*}[u,A'']+\mathcal{F}_{*}[u,A\setminus K]
	\leq
	\mathcal{F}_{*}[u,A'']+\varepsilon
	\]
	Letting $\varepsilon\downarrow 0$ gives~\eqref{eq:inequality-measure-property2}.

	\textit{Step 3.} Let $A, B$ be open subsets of $\Omega$. We now prove that
	\begin{equation}\label{eq:inequality-measure-property3}
	\mathcal{F}_{*}[u,A\cup B]
	\leq
	\mathcal{F}_{*}[u,A]+\mathcal{F}_{*}[u,B].
	\end{equation}
	Fix $\varepsilon>0$. By Step 2 there exists an open set $U\Subset A\cup B$ such that
	\[
	\mathcal{F}_{*}[u,A\cup B]-\varepsilon
	\leq
	\mathcal{F}_{*}[u,U].
	\]
	Choose $A'\Subset A$ open such that $U\subset A'\cup B$. By Step 1 we have
	\[
	\mathcal{F}_{*}[u,A\cup B]-\varepsilon
	\leq
	\mathcal{F}_{*}[u,A'\cup B]
	\leq
	\mathcal{F}_{*}[u,A]+\mathcal{F}_{*}[u,B].
	\]
	Letting $\varepsilon\downarrow 0$ yields~\eqref{eq:inequality-measure-property3}.

	\textit{Step 4.} Finally, we prove that for open sets $A, B$ such that $A\cap B=\emptyset$ the inequality
	\begin{equation}\label{eq:inequality-measure-property4}
	\mathcal{F}_{*}[u,A\cup B]
	\geq
	\mathcal{F}_{*}[u,A]+\mathcal{F}_{*}[u,B]
	\end{equation}
	holds.

	We choose a sequence $(u_h)\subset\LU(A\cup B)$ converging weakly* to $u\in\UU(A\cup B)$ and such that
	\[
	\lim_{h\to\infty}\mathcal{F}[u_h,A\cup B]=\mathcal{F}_{*}[u,A\cup B].
	\]
	Since the sets $A$ and $B$ are disjoint, we have
	\begin{align*}
	\mathcal{F}_{*}[u,A\cup B]
	&=
	\lim_{h\to\infty}\mathcal{F}[u_h,A\cup B] \\
	&\geq
	\liminf_{h\to\infty}\mathcal{F}[u_h,A]
	+
	\liminf_{h\to\infty}\mathcal{F}[u_h,B] \\
	&\geq
	\mathcal{F}_{*}[u,A]+\mathcal{F}_{*}[u,B],
	\end{align*}
	hence we proved~\eqref{eq:inequality-measure-property4}. By Theorem~\ref{thm:degiorgi-letta} we infer that the set function $V\mapsto\mathcal{F}_{*}[u,V]$ is a restriction to open sets of a finite Radon measure. \qedhere
\end{proof}

\begin{proposition}[Upper estimate]\label{prop:upper-bound}
	The relaxation $\mathcal{F}_{*}$ satisfies the upper bound
	\[
	\mathcal{F}_{*}[u,\Omega]
	\leq
	\int_{\Omega}f(\mathcal{E}u)\,\dif x
	+
	\int_{\Omega}f_{\dev}^{\#}\left(\frac{\dif E^su}{\dif|E^su|}\right)\,\dif|E^su|
	.
	\]
\end{proposition}
\begin{proof}
	By Remark~\ref{rem:area-strict-continuous-extension} we can find a sequence $(u_h)\subset\LU(\Omega)\cap\CC^\infty(\Omega;\mathbb{R}^d)$ converging area-strictly to $u\in\UU(\Omega)$. Since the area-strict convergence is stronger than the weak* convergence, by the definition of $\mathcal{F}_{*}$, it follows that
	\[
	\mathcal{F}_{*}[u,\Omega]
	\leq
	\liminf_{h\to\infty}\mathcal{F}[u_h,\Omega]
	=
	\int_{\Omega}f(\mathcal{E}u)\,\dif x
	+
	\int_{\Omega}f_{\dev}^{\#}\left(\frac{\dif E^su}{\dif|E^su|}\right)\,\dif|E^su|,
	\]
	where the equality follows from Remark~\ref{rem:area-strict-continuous-extension}.
\end{proof}

The conclusion of Theorem~\ref{thm:main-result} will follow once we prove the lower bound.

\begin{proposition}[Lower estimate]\label{prop:lower-bound}
	For $u\in\UU(\Omega)$ the inequality
	\[
	\mathcal{F}_{*}[u,\Omega]
	\geq
	\int_{\Omega}f(\mathcal{E}u)\,\dif x
	+
	\int_{\Omega}f_{\dev}^{\#}\left(\frac{\dif E^su}{\dif|E^su|}\right)\,\dif|E^su|
	\]
	holds.
\end{proposition}
\begin{proof}
	We treat separately $\mathcal{L}^d$-a.e.\ regular point $x_0\in\Omega$ and $|E^su|$-a.e.\ singular point $x_0\in\Omega$.

	\textit{Regular points.} The proof is based on a blow-up argument. Fix $x_0\in\Omega$ such that
	\begin{enumerate}
		\item $u$ is approximately differentiable at $x_0$,
		\item $\displaystyle\lim_{r\downarrow 0}\frac{|Eu|(B(x_0,r))}{\omega_dr^d}=\frac{\dif|Eu|}{\dif\mathcal{L}^d}(x_0)=|\mathcal{E}u(x_0)|$,
		\item $x_0$ is an $\mathcal{L}^d$-Lebesgue point of $\div u$.
	\end{enumerate}
	Since $u\in\UU(\Omega)$, these properties hold for $\mathcal{L}^d$-almost every $x\in\Omega$. In particular (1) is a consequence of Theorem~7.4 in~\cite{AmbrosioCosciaDalMaso97}, whereas (2) follows from Theorem~\ref{thm:besicovitch}. For $y\in B(0,1)$ define maps
	\[
	u_r(y):=\frac{u(x_0+ry)-\tilde{u}(x_0)}{r}, \quad 0<r<\dist(x_0,\partial\Omega),
	\]
	where $\tilde{u}$ is the precise representative of $u$. For $u_0(y):=\nabla u(x_0)y$ we have the strong convergence $u_r\to u_0$ in $\LL^1(B(0,1);\mathbb{R}^d)$. Indeed, by the approximate differentiability we have
	\begin{multline*}
	\int_{B(0,1)}|u_r(y)-u_0(y)|\,\dif y
	=
	\frac{1}{r^d}\int_{B(x_0,r)}\left| \frac{u(z)-\tilde{u}(x_0)-\nabla u(x_0)(z-x_0)}{r} \right|\,\dif z\to 0
	\end{multline*}
	as $r\downarrow 0$. Moreover, we have strict convergence:
	\begin{align*}
	\lim_{r\downarrow 0}|Eu_r|(B(0,1))
	&=
	\omega_d\lim_{r\downarrow 0}\frac{|Eu|(B(x_0,r))}{\omega_dr^d} \\
	&=
	\omega_d|\mathcal{E}u(x_0)| \\
	&=
	|Eu_0|(B(0,1)),
	\end{align*}
	thus $(u_r)$ is bounded in $\BD(B(0,1))$. Note also that for $\varphi\in\LL^2(B(0,1))$ we have
	\begin{align*}
	&\left|\int_{B(0,1)}\varphi(y)(\div u_r(y)-\div u_0(y))\,\dif y\right| \\
	&\qquad\leq
	\norm{\varphi}_2\int_{B(0,1)}|\div u(x_0+ry)-\div u(x_0)|^2\,\dif y \\
	&\qquad=
	\omega_d\norm{\varphi}_2\mint{-}_{B(x_0,r)}|\div u(z)-\div u(x_0)|^2\,\dif z.
	\end{align*}
	The right-hand side vanishes as $r\downarrow 0$ by the Lebesgue point property (3). Hence, $u_r\starconv u_0$ weakly* in $\UU(B(0,1))$. In virtue of Lemma~\ref{lem:lsc-linear} and Proposition~\ref{prop:relaxation-ts} and the scaling properties of $\mathcal{F}_{*}$ we obtain
	\begin{align*}
	\liminf_{r\downarrow 0}\frac{\mathcal{F}_{*}[u,B(x_0,r)]}{r^d}
	&=
	\liminf_{r\downarrow 0}\mathcal{F}_{*}[u_r,B(0,1)] \\
	&\geq
	\mathcal{F}_{*}[u_0,B(0,1)]\\
	&\geq \int_{B(0,1)}f(\mathcal{E}u_0(y))\,\dif y \\
	&=
	\omega_d f(\mathcal{E}u(x_0)).
	\end{align*}
	Therefore, by Lemma~\ref{lem:measure-property} and Proposition~\ref{prop:lowerHausdorffBound} we obtain
	\[
	\mathcal{F}_{*}^a[u,B]\geq\int_B f(\mathcal{E}u)\,\dif x
	\]
	for any Borel set $B\subset\Omega$.

	\textit{Singular points.} We want to prove that for all Borel sets $B\subset\Omega$ the inequality
	\[
	\mathcal{F}_{*}^s[u,B]\geq\int_B f_{\dev}^{\#}\left(\frac{\dif E^su}{\dif|E^su|}\right)\,\dif|E^su|
	\]
	holds. We fix $x_0\in\Omega$ such that
	\[
	\frac{\dif E^su}{\dif|E^su|}(x_0)=a\odot b, \qquad a,b\in\mathbb{R}^d\setminus\{0\}, \ a\perp b.
	\]
	This property holds for $|E^su|$-a.e.\ $x_0\in\Omega$ by Theorem~\ref{thm:bd-rank-one}. It suffices to establish the inequality
	\[
	\lim_{r\downarrow 0}\frac{\mathcal{F}_{*}[u,B(x_0,r)]}{|Eu|(B(x_0,r))}\geq f_{\dev}^{\#}(a\odot b)
	\]
	at any point $x_0\in\Omega$ for which the limit on the left-hand side exists, which is the case at $|Eu|$-almost every $x_0$ (cf.~Corollary~2.23 in \cite{AmbrosioFuscoPallara00} with $\mu=|Eu|$). By the coercivity of $\mathcal{F}$ and a diagonal argument similar to the one contained in the proof of Proposition~\ref{prop:relaxation-ts}, we can choose a sequence  $(u_h)\subset\LU(B(x_0,r))$ such that $ u_h\starconv u $ weakly* in $ \UU(B(x_0,r))$ and
	\[
	\lim_{h\to\infty}\mathcal{F}[u_h,B(x_0,r)]=\mathcal{F}_{*}[u,B(x_0,r)].
	\]
	We then have
	\begin{align*}
	&\mathcal{F}_{*}[u,B(x_0,r)] \\
	&\quad=
	\lim_{h\to\infty}\mathcal{F}[u_h,B(x_0,r)] \\
	&\quad=
	\lim_{h\to\infty}\int_{B(x_0,r)}f(\mathcal{E}u_h)\,\dif x \\
	&\quad=
	\lim_{h\to\infty}\int_{B(x_0,r)}f(\mathcal{E}u_h)-f_{\dev}^{\#}(\dev\mathcal{E}u_h)\,\dif x+\int_{B(x_0,r)}f_{\dev}^{\#}(\dev\mathcal{E}u_h)\,\dif x \\
	&\quad=:
	\lim_{h\to\infty}\left(I^{(1)}_{h,r}+I^{(2)}_{h,r}\right).
	\end{align*}
	In virtue of~\eqref{eq:subcritical-growth} we have
	\begin{align*}
	I^{(1)}_{h,r}
	&=
	\int_{B(x_0,r)}f(\mathcal{E}u_h)-f_{\dev}^{\#}(\dev\mathcal{E}u_h)\,\dif x \\
	&\geq
	-M\int_{B(x_0,r)}1+|\div u_h|^\gamma+|\dev\mathcal{E}u_h|^\delta \,\dif x
	\end{align*}
	for $\gamma\in[0,2)$ and $\delta\in[0,1)$.
	We can assume that
	\[
	|\dev\mathcal{E}u_h|^\delta \rightharpoonup \xi \quad\text{weakly in } \LL^{1/\delta}(B(x_0,r))
	\]
	for some $\xi\in\LL^{1/\delta}(B(x_0,r))$.

	For $0\leq\gamma<2$ by H{\"o}lder's inequality we obtain
	\[
	\int_{B(x_0,r)}|\div u_h|^\gamma\,\dif x
	\leq
	\sup_h\,\norm{\div u_h}_2^{\gamma}\,|B(x_0,r)|^{1-\gamma/2}.
	\]
	Thus,
	\begin{align*}
	\lim_{h\to\infty} I^{(1)}_{h,r}
	\geq
	- C_{M,\gamma}\left(|B(x_0,r)|^{1-\gamma/2}+|B(x_0,r)|+\int_{B(x_0,r)} \xi \,\dif x\right).
	\end{align*}
	
	Therefore
	\[
	\lim_{r\downarrow 0} \lim_{h\to\infty} \frac{I^{(1)}_{h,r}}{|Eu|(B(x_0,r))}
	\geq 0.
	\]

	By Proposition~\ref{prop:span-of-SD} the set
	\[
	\mathcal{S}:=\left\{a\odot b: \ a,b\in\mathbb{R}^d, \ a\cdot b=0 \right\}
	\]
	spans the space of symmetric and deviatoric matrices $\SD$. Moreover, the recession function $f_{\dev}^{\#}$ is positively 1-homogeneous and convex at points of $\mathcal{S}$ (see Remark~\ref{rem:propertiesOfRecessionF}). In virtue of Theorem~\ref{thm:kk-convexity} for each orthogonal $a,b\in\mathbb{R}^d$ there exists a linear function $\ell:\SD\to\mathbb{R}$ such that $f_{\dev}^{\#}(D)\geq \ell(D)$ for all $D\in\SD$ and $f_{\dev}^{\#}(a\odot b)=\ell(a\odot b)$. For all but finitely many $r>0$ we can assume that $\Lambda(\partial B(x,r))=0$, where $\Lambda\in\Mpos(\Omega)$ is the weak* limit of (a subsequence of) the measures $|\ell(\dev(\mathcal{E}u_h))|\mathcal{L}^d$.
	Therefore, we have
	\begin{align*}
	\lim_{h\to\infty}I^{(2)}_{h,r}
	&=
	\lim_{h\to\infty}\int_{B(x_0,r)}f_{\dev}^{\#}(\dev\mathcal{E}u_h)\,\dif x \\
	&\geq
	\limsup_{h\to\infty}\int_{B(x_0,r)}\ell(\dev\mathcal{E}u_h)\,\dif x \\
	&=
	\ell(\dev Eu(B(x_0,r))),
	\end{align*}
	where the last equality follows from the linearity of $\ell$ and Proposition~1.62(b) in~\cite{AmbrosioFuscoPallara00}. Combining the above estimates yields
	\begin{align*}
	\lim_{r\downarrow 0}\frac{\mathcal{F}_{*}[u,B(x_0,r)]}{|Eu|(B(x_0,r))}
	&\geq
	\limsup_{r\downarrow 0}\frac{\ell(\dev Eu(B(x_0,r)))}{|Eu|(B(x_0,r))} \\
	&=
	\limsup_{r\downarrow 0}\ell\left(\dev\left(\frac{ Eu(B(x_0,r))}{|Eu|(B(x_0,r))}\right)\right) \\
	&=
	\ell\left(\dev\left(\lim_{r\downarrow 0}\frac{Eu(B(x_0,r))}{|Eu|(B(x_0,r))}\right)\right) \\
	&=
	\ell\left(\dev(a\odot b)\right) \\
	&=
	\ell(a\odot b) \\
	&=
	f_{\dev}^{\#}(a\odot b).
	\end{align*}
	This finishes the proof.\qedhere
\end{proof}

\section{Relaxation in \texorpdfstring{$\BD$}{BD}}\label{sec:relaxation-bd}
In this section we prove Theorem~\ref{thm:lsc-bd}. The strategy of the proof is effectively the same as the one used for Theorem~\ref{thm:main-result}, except that we may prove the lower bound at singular points without using the Kirchheim--Kristensen Theorem~\ref{thm:kk-convexity}. In fact, the $\BD$ counterparts of our auxiliary results are substantially easier to establish than in the mixed-growth case, so we omit their proofs.

In all of the following we assume that $f$ is already symmetric-quasiconvex. This is no restriction since an inspection of the proof of the main result in~\cite{BarrosoFonsecaToader00}, Theorem 3.5, yields that the relaxation of the functional
\[
\int_\Omega f(\mathcal{E}u) \,\dif x
\]
for all $u \in \LD(\Omega)$ is given by
\[
\int_\Omega (SQf)(\mathcal{E}u) \,\dif x,
\]
without any restriction on the recession function (the condition~(3.2) in~\cite{BarrosoFonsecaToader00} is only used for the jump part).

We have the following analogue of Lemma~\ref{lem:lsc-linear} (note that there is a $\BD(\Omega)$-analogue of Lemma~\ref{lem:convInU}, see the remark after that lemma).

\begin{lemma}\label{lem:linear-bd}
	Let $A\in\mathbb{R}^{d\times d}_{\mathrm{sym}}$ and let $(u_h)\subset\BD(\Omega)$ be a sequence such that $u_h\starconv Ax$ weakly* in $\BD(\Omega)$. Then
	\begin{equation}\label{eq:linear-bd}
	|\Omega|f(A)\leq\liminf_{h\to\infty}\int_{\Omega}f(\mathcal{E}u_h)\,\dif x.
	\end{equation}
\end{lemma}

Since the topology of weak* convergence in $\BD(\Omega)$ is metrizable on bounded sets, it follows that the relaxation $\mathcal{F}_{*}$, as defined in~\eqref{eq:relaxBD}, is lower semicontinuous with respect to this topology (cf.~\cite{Attouch14} for details).

In the remaining part of this section we will establish an integral representation for $\mathcal{F}_{*}$, that is
\begin{equation}\label{eq:representation-bd}
\mathcal{F}_{*}[u,\Omega]
=
\int_{\Omega}f(\mathcal{E}u)\,\dif x
+
\int_{\Omega}f^{\#}\left(\frac{\dif E^su}{\dif|E^su|}\right)\,\dif|E^su|.
\end{equation}
More specifically, we will establish the upper and the lower estimate on the relaxation $\mathcal{F}_{*}$ by the right-hand side of~\eqref{eq:representation-bd}. We begin with the upper estimate.

Let us denote by $\mathrm{D}(\mathbb{R}^{d\times d}_{\mathrm{sym}})$ the class of continuous functions $f:\mathbb{R}^{d\times d}_{\mathrm{sym}}\to\mathbb{R}$ with linear growth at infinity and for which the strong recession function
\[
f^\infty(A):=\lim_{A'\to A, \ s\to\infty}\frac{f(sA')}{s}
\]
exists. For such functions we have the following continuity result.
\begin{theorem}[Reshetnyak~\cite{KristensenRindler10Relax}]\label{thm:reshetnyak}
	Let $(\mu_h)\subset\M(\Omega;\mathbb{R}^d)$ be a sequence of measures, such that $ \mu_h\to\mu $ area-strictly for some $\mu\in\M(\Omega;\mathbb{R}^d)$. Then, for $f\in\mathrm{D}(\mathbb{R}^{d\times d}_{\mathrm{sym}})$ it holds that
	\begin{equation*}
	\int_{\Omega}f\left(\frac{\dif\mu_h}{\dif\mathcal{L}^d}\right)\,\dif x+\int_{\Omega}f^\infty\left(\frac{\dif\mu_h^s}{\dif|\mu_h^s|}\right)\,\dif|\mu_h^s|
	\to
	\int_{\Omega}f\left(\frac{\dif\mu}{\dif\mathcal{L}^d}\right)\,\dif x+\int_{\Omega}f^\infty\left(\frac{\dif\mu^s}{\dif|\mu^s|}\right)\,\dif|\mu^s|
	\end{equation*}
	as $h\to\infty$.
\end{theorem}
Furthermore, it turns out that the admissible integrands $f$ in Theorem~\ref{thm:lsc-bd} can be approximated by functions in $\mathrm{D}(\mathbb{R}^{d\times d}_{\mathrm{sym}})$ (cf.~\cite[Lemma~2.2]{KristensenRindler10YM}).
\begin{lemma}[Pointwise approximation]\label{lem:pointwise-approximation}
	For every continuous function $f:\mathbb{R}^{d\times d}_{\mathrm{sym}}\to\mathbb{R}$ with a linear growth at infinity, there exists a decreasing sequence $(f_k)\subset\mathbf{D}(\mathbb{R}^{d\times d}_{\mathrm{sym}})$, such that
	\[
	\inf_k f_k=\lim_{k\to\infty}f_k=f \quad \text{and} \quad \inf_k f_k^\infty=\lim_{k\to\infty}f_k^\infty=f^{\#},
	\]
	with pointwise convergence.
\end{lemma}

We are now ready to establish the upper bound.
\begin{lemma}[Upper estimate]\label{lem:upper-bound-bd}
	For $u\in\BD(\Omega)$ the inequality
	\[
	\mathcal{F}_{*}[u,\Omega]
	\leq
	\int_{\Omega}f(\mathcal{E}u)\,\dif x
	+
	\int_{\Omega}f^{\#}\left(\frac{\dif E^su}{\dif|E^su|}\right)\,\dif|E^su|.
	\]
	holds.
\end{lemma}
\begin{proof}
	Fix $u\in\BD(\Omega)$. There exists a sequence $(u_h)\subset\LD(\Omega)\cap\CC^\infty(\Omega;\mathbb{R}^d)$, such that $u_h\to u$ area-strictly (cf.~\cite[Theorem~14.1.1]{Attouch14}). Let $(f_k)\subset\mathrm{D}(\mathbb{R}^{d\times d}_{\mathrm{sym}})$ be a sequence as in Lemma~\ref{lem:pointwise-approximation}. By Theorem~\ref{thm:reshetnyak} we have for each $k\in\mathbb{N}$:
	\begin{equation*}
	\lim_{h\to\infty}\int_{\Omega}f_k(\mathcal{E}u_h)\,\dif x
	=
	\int_{\Omega}f_k(\mathcal{E}u)\,\dif x
	+
	\int_{\Omega}f_k^\infty\left(\frac{\dif E^su}{\dif|E^su|}\right)\,\dif|E^su|.
	\end{equation*}
	Hence,
	\begin{equation*}
	\liminf_{h\to\infty}\int_{\Omega}f(\mathcal{E}u_h)\,\dif x
	\leq
	\int_{\Omega}f_k(\mathcal{E}u)\,\dif x
	+
	\int_{\Omega}f_k^\infty\left(\frac{\dif E^su}{\dif|E^su|}\right)\,\dif|E^su|.
	\end{equation*}
	Since the area-strict convergence is stronger than the weak* convergence, by the definition of $\mathcal{F}_{*}$, it follows that
	\[
	\mathcal{F}_{*}[u,\Omega]
	\leq
	\liminf_{h\to\infty}\int_{\Omega}f(\mathcal{E}u_h)\,\dif x
	\leq
	\int_{\Omega}f_k(\mathcal{E}u)\,\dif x
	+
	\int_{\Omega}f_k^\infty\left(\frac{\dif E^su}{\dif|E^su|}\right)\,\dif|E^su|.
	\]
	By the monotone convergence theorem, letting $k\to\infty$ ends the proof. \qedhere
\end{proof}

As in the proof of Theorem~\ref{thm:main-result}, to prove the lower estimate, we first prove that for a given $u\in\BD(\Omega)$ the map $V\mapsto\mathcal{F}_{*}[u,V]$ is the restriction to the open subsets of $\Omega$ of some Radon measure, which we still denote by $\mathcal{F}_{*}[u,\cdot]$. Then, we decompose this measure into the absolutely continuous and singular parts with respect to the Lebesgue measure, i.e.
\[
\mathcal{F}_{*}[u,\cdot]=\mathcal{F}_{*}^a[u,\cdot]+\mathcal{F}_{*}^s[u,\cdot], \quad \mathcal{F}_{*}^a[u,\cdot] \ll \mathcal{L}^d\mres\Omega, \quad \mathcal{F}_{*}^s[u,\cdot]\perp\mathcal{L}^d\mres\Omega
\]
and then prove that
\[
\mathcal{F}_{*}^a[u,B]\geq\int_B f(\mathcal{E}u)\,\dif x
\quad\text{and}\quad
\mathcal{F}_{*}^s[u,B]\geq\int_{B}f^{\#}\left(\frac{\dif E^su}{\dif|E^su|}\right)\dif|E^su|
\]
for any Borel set $B\subset\Omega$.

\begin{lemma}\label{lem:measure-property-bd}
	For all $u\in\BD(\Omega)$ the set function $V\mapsto\mathcal{F}_{*}[u,V]$ is a restriction to the open subsets of $\Omega$ of a finite Radon measure.
\end{lemma}
The proof of Lemma~\ref{lem:measure-property-bd} is a straightforward adaptation of Lemma~\ref{lem:measure-property} so we omit the details here.

\begin{remark}\label{rem:invariances-bd}
	The relaxation $\mathcal{F}_{*}$ satisfies the same invariant properties as in Remark~\ref{rem:invariances}.
\end{remark}

\begin{lemma}\label{lem:lower-est}
	Let $Q$ be an open $d$-cube with side length 1 and faces either parallel or orthogonal to $a$, let  $v\in\BD(Q)$ be representable in $Q$ as
	\[
	v(y):=g(y\cdot a)b+c(a\otimes b)y+Wy+\bar{v},
	\]
	where $g:\mathbb{R}\to\mathbb{R}$ is a locally bounded and increasing function, $a,b\in\mathbb{R}^d\setminus\{0\}$, $c>0$, $W\in\mathbb{R}^{d\times d}_{\mathrm{skew}}$ and $\bar{v}\in\mathbb{R}^d$. Let $u\in\BD(Q)$ be such that $\supp(u-v)\Subset Q$. Then,
	\[
	\mathcal{F}_{*}[u,Q]\geq f(Eu(Q)).
	\]
\end{lemma}
\begin{proof}
	We only treat the case where $a,b$ are not parallel. The case $a,b$ parallel is in fact easier.
	In virtue of the above remark, we may without loss of generality further assume that $a=e_1$, $b=e_2$ and $Q=(0,1)^d$. Then,
	\[
	v(y)=g(y_1)e_2+cy_2e_1+Wy+\bar{v}.
	\]
	Let
	\[
	q:=|Dg|(0,1)=g(1^-)-g(0^+).
	\]
	Since $u\in\BD(Q)$, the function
	\[
	w(x):=u(x-\lfloor x\rfloor)+qe_2\lfloor x_1\rfloor+ce_1\lfloor x_2\rfloor+Wx+\bar{v}, \quad x\in\mathbb{R}^d,
	\]
	is in $\BD_{\mathrm{loc}}(\mathbb{R}^d)$. Here the floor function of a vector $x\in\mathbb{R}^d$ is understood component-wise. Let $u_h(y):=w(hy)/h$, $h \in \mathbb{N}$. For
	\[
  	u_0(y):=qe_2y_1+ce_1y_2+Wy+\bar{v}
	\]
	it holds that
	\begin{align*}
	&\int_Q |u_h(y)-u_0(y)|\,\dif y \\
	&\quad=
	\frac{1}{h}\int_Q |u(hy-\lfloor hy\rfloor)-qe_2(hy_1-\lfloor hy_1\rfloor)-ce_1(hy_2-\lfloor hy_2\rfloor)|\,\dif y \\
	&\quad=
	\frac{1}{h^{d+1}}\int_{(0,h)^d}|u(x-\lfloor x\rfloor)-qe_2(x_1-\lfloor x_1\rfloor)-ce_1(x_2-\lfloor x_2\rfloor)|\,\dif x \\
	&\quad=
	\frac{1}{h}\int_Q|w(y)-(qe_2y_1+ce_1y_2+Wy+\bar{v})|\,\dif y,
	\end{align*}
	hence $u_h\to u_0$ in $\LL^1(Q;\mathbb{R}^d)$. The sequence $(u_h)$ is uniformly norm-bounded in $\BD(Q)$, so we also have that $ u_h\starconv u_0$ weakly* in $\BD(Q)$ (the argument is the same as in Lemma~\ref{lem:convInU}).

	Let $Q_1,\ldots,Q_{h^d}$ be the canonical decomposition of $Q$ into open cubes with sides parallel to those of $Q$ and side length $1/h$. Then, by the scaling property of $\mathcal{F}_{*}$, for all $i=1,\ldots,h^d$ it holds that
	\[
	\mathcal{F}_{*}[u_h,Q_i]=\mathcal{F}_{*}[u_h,(0,1/h)^d]=h^{-d}\mathcal{F}_{*}[u,Q].
	\]
	Moreover, since $\supp(u-v)\Subset Q$, the measure $|Ew|$ vanishes on every hyperplane of the form $x_j=k$, with $k\in\mathbb{Z}$, $j=1,\ldots,d$. Thus we have that $|Eu_h|(Q\cap\partial Q_i)=0$ for all $i=1,\ldots,h^d$. Note that for any open set $A\subset Q$ the inequality
	\[
	\mathcal{F}_{*}[u,A]
	\leq 
	M\left(\mathcal{L}^d(A)+|Eu|(A)\right)
	\]
	holds as a consequence of the linear growth of the integrand and the density of smooth functions with respect to the strict convergence. By the regularity of measures this inequality can then be extended to any Borel set, hence
	\[
	\mathcal{F}_{*}[u,Q\cap\partial Q_i]=0.
	\]
	Therefore, for any $h\in\mathbb{N}$ we obtain
	\[
	\mathcal{F}_{*}[u_h,Q]=\sum_{i=1}^{h^d}\mathcal{F}_{*}[u_h,Q_i]=\sum_{i=1}^{h^d}h^{-d}\mathcal{F}_{*}[u,Q]=\mathcal{F}_{*}[u,Q].
	\]
	By the weak* lower semicontinuity of $\mathcal{F}_{*}$ we obtain
	\[
	\mathcal{F}_{*}[u,Q]
	=
	\lim_{h\to\infty}\mathcal{F}_{*}[u_h,Q]
	\geq
	\mathcal{F}_{*}[u_0,Q].
	\]
	Let $S\in\mathbb{R}^{d\times d}_{\mathrm{skew}}$ be the skew-symmetric matrix defined as
	\[
	S:=\frac{q-c}{2}(e_1\otimes e_2-e_2\otimes e_1).
	\]
	Then, by Remark~\ref{rem:invariances-bd} we obtain
	\begin{align*}
	\mathcal{F}_{*}[u_0,Q]
	&=
	\mathcal{F}_{*}[q(e_2\otimes e_1)y+c(e_1\otimes e_2)y+Wy+\bar{v},Q] \\
	&=
	\mathcal{F}_{*}[q(e_2\otimes e_1)y+c(e_1\otimes e_2)y+Sy,Q] \\
	&=
	\mathcal{F}_{*}[(q+c)(e_1\odot e_2)y,Q].
	\end{align*}
	In virtue of Lemma~\ref{lem:linear-bd}, for every $(v_h)\subset\LD(Q)$ such that $v_h\starconv (q+c)(e_1\odot e_2)y$ weakly* in $\BD(Q)$ it holds that
	\[
	\liminf_{h\to\infty}\mathcal{F}[v_h,Q]\geq\mathcal{F}[(q+c)(e_1\odot e_2)y,Q].
	\]
	Taking the infimum over all such sequences yields
	\[
	\mathcal{F}_{*}[(q+c)(e_1\odot e_2)y,Q]\geq \mathcal{F}[(q+c)(e_1\odot e_2)y,Q].
	\]
	Since $Eu(Q)=Ev(Q)$ and, by the definition of $q$, $Ev(Q)=Eu_0(Q)=(q+c)(e_1\odot e_2)$, we can write
	\[
	\mathcal{F}_{*}[u,Q]
	\geq
	\mathcal{F}_{*}[u_0,Q]
	\geq
	\mathcal{F}[(q+c)(e_1\odot e_2)y,Q]
	=
	f(Eu(Q)).
	\]
	This proves the lemma. \qedhere
\end{proof}

\begin{lemma}[Lower estimate]\label{lem:lower-bound}
	For $u\in\BD(\Omega)$ the inequality
	\[
	\mathcal{F}_{*}[u,\Omega]
	\geq
	\int_{\Omega}f(\mathcal{E}u)\,\dif x
	+
	\int_{\Omega}f^{\#}\left(\frac{\dif E^su}{\dif|E^su|}\right)\dif|E^su|
	\]
	holds.
\end{lemma}
\begin{proof} We treat separately $\mathcal{L}^d$-a.e.\ regular point $x_0\in\Omega$ and $|E^su|$-a.e.\ singular point $x_0\in\Omega$.

	\textit{Regular points.} For regular points, the argument is exactly the same as in the first part of the proof of Proposition~\ref{prop:lower-bound}.

	\textit{Singular points.} We want to prove that for all Borel sets $B\subset\Omega$ the inequality
	\[
	\mathcal{F}_{*}^s[u,B]\geq\int_B f^{\#}\left(\frac{\dif E^su}{\dif|E^su|}\right)\dif|E^su|
	\]
	holds. In order to do that we fix $x_0\in\Omega$ such that
	\begin{enumerate}
		\item\label{enum:alberti-type} $\begin{aligned}[t] \frac{\dif E^su}{\dif|E^su|}(x_0)=a\odot b \end{aligned}$
		for some $a,b\in\mathbb{R}^{d}\setminus\{0\}$;
		\item\label{enum:besicovitch-conclusion} $\alpha_r:=r^{-d}|Eu|(Q(x_0,r))\to\infty$ as $r\downarrow 0$, where $Q(x_0,r):=x_0+rQ$ and $Q$ is a (fixed) open $d$-cube with center 0, side-length 1 and sides either parallel or orthogonal to~$a$.
	\end{enumerate}
	These properties hold for $|E^su|$-a.e.\ $x_0\in\Omega$ in virtue of Theorem~\ref{thm:bd-rank-one} and Theorem~\ref{thm:besicovitch}. As in the proof of Proposition \ref{prop:lower-bound}, it suffices to establish the inequality
	\[
	\lim_{r\downarrow 0}\frac{\mathcal{F}_{*}[u,Q(x_0,r)]}{|Eu|(Q(x_0,r))}\geq f^{\#}(a\odot b)
	\]
	for all $|Eu|$-Lebesgue points $x_0\in\Omega$ such that the limit on the left-hand side exists (cf.~Corollary~2.23 in \cite{AmbrosioFuscoPallara00} with $\mu=|Eu|$).

	Define a blow-up sequence
	\[
	v_r(y):=\frac{u(x_0+ry)-[u]_{Q(x_0,r)}}{r\alpha_r}+R_r(y), \quad y\in Q, \quad 0<r<\dist(x_0,\partial\Omega),
	\]
	where $R_r:\mathbb{R}^d\to\mathbb{R}^d$ is a family of rigid deformations and $[u]_{Q(x_0,r)}:=\mint{-}_{Q(x_0,r)}u\,\dif x$ is the average of $u$ over $Q(x_0,r)$.

	In virtue of Lemma~2.14 in~\cite{DePhilippisRindler16sec}, up to a subsequence, the blow-up sequence $(v_r)$ converges weakly* in $\BD(Q)$ to the function
	\[
	v_0(y):=h(y\cdot a)b+c(a\otimes b)y+Wy+\bar{v},
	\]
	with a bounded and increasing function $h:(-1/2,1/2)\to\mathbb{R}$, $c>0$, and a rigid deformation $Wy+\bar{v}$, where $W \in \mathbb{R}^{d\times d}_\mathrm{skew}$, $\bar{v}\in\mathbb{R}^d$.

	Note that for any Borel set $B\subset Q$ we have
	\begin{equation}\label{eq:Evr}
	Ev_r(B)=\frac{r^{1-d}Eu(x_0+rB)}{r\alpha_r}=\frac{Eu(x_0+rB)}{|Eu|(Q(x_0,r))}
	\end{equation}
	hence $|Ev_r|(Q)=1$. Consequently, by Proposition~1.62(b) in~\cite{AmbrosioFuscoPallara00}, we also have $|Ev_0|(Q)\leq 1$.

	Fix $0<t<1$ and let $Q_t:=tQ$ be a re-scaled cube. There exists a (not particularly labeled) sequence of radii $r \downarrow 0$ such that
	\begin{equation}\label{eq:technical}
	\lim_{r\downarrow 0}\frac{|Eu|(Q(x_0,tr))}{|Eu|(Q(x_0,r))}\geq t^d.
	\end{equation}
	Indeed, if it was not true, then for some $0<t_0<1$ we could find $0<r_0<1$ such that
	\[
	|Eu|(Q(x_0,t_0r))\leq t_0^d|Eu|(Q(x_0,r))
	\]
	for all $r<r_0$. Iterating the above inequality yields:
	\[
	|Eu|(Q(x_0,t_0^kr_0))\leq t_0^{kd}|Eu|(Q(x_0,r_0))
	\]
	for all $k\in\mathbb{N}$. Since any $0<r<r_0$ is in the interval $(t_0^{k+1}r_0,t_0^kr_0]$ for some $k\in\mathbb{N}$ we obtain
	\[
	|Eu|(Q(x_0,r))
	\leq
	|Eu|(Q(x_0,t_0^kr_0))
	\leq
	t_0^{kd}|Eu|(Q(x_0,r_0))
	\leq
	\frac{|Eu|(Q(x_0,r_0))}{t_0^dr_0^d}r^d.
	\]
	Hence for any $0<r<r_0$,
	\[
	\alpha_r\leq \frac{|Eu|(Q(x_0,r_0))}{t_0^dr_0^d},
	\]
	which is a contradiction, since $\alpha_r\to+\infty$ as $r\downarrow 0$. So,~\eqref{eq:technical} follows.

	Note that~\eqref{eq:technical} yields
	\begin{equation}\label{eq:technical2}
	\lim_{r\downarrow 0}|Ev_r|(\overline{Q}_t)\geq t^d.
	\end{equation}
	Then, for any weak* limit $\nu$ of $|Ev_r|$ in $Q$ we get by Example~1.63 in~\cite{AmbrosioFuscoPallara00} that $\nu(\overline{Q}_t) \geq t^d$. On the other hand, $Ev_r \starconv Ev_0$ and $Ev_0(Q) = \frac{a \odot b}{|a \odot b|} \nu(Q)$ by Theorem~\ref{thm:besicovitch},~\eqref{eq:Evr}, and~\eqref{enum:alberti-type}. Moreover, $|Ev_0|(Q) \leq \nu(Q) = |Ev_0(Q)| \leq |Ev_0|(Q)$, hence, together with $\nu \geq |Ev_0|$ we obtain that $\nu = |Ev_0|$ on $Q$. Thus, $|Ev_0|(\overline{Q}_t)\geq t^d$.

	Define $w_r:=\varphi v_r+(1-\varphi)v_0$, where $\varphi\in\CC_c^1(Q;[0,1])$ with $\varphi\equiv 1$ on a neighborhood of $\overline{Q}_t$. Clearly, the sequence $(w_r)$ converges to $v_0$ strongly in $\LL^1(Q;\mathbb{R}^d)$ and
	\begin{align*}
	|E(w_r-v_r)|(Q)
	&\leq
	|E(v_r-v_0)|(Q\setminus \overline{Q}_t)+\int_Q |\nabla\varphi|\,|v_r-v_0|\,\dif y \\
	&\leq
	|Ev_r|(Q\setminus \overline{Q}_t)+|Ev_0|(Q\setminus \overline{Q}_t)+\int_Q |\nabla\varphi|\,|v_r-v_0|\,\dif y.
	\end{align*}
	Therefore, by~\eqref{eq:technical2}, we have
	\[
	\limsup_{r\downarrow 0}|E(w_r-v_r)|(Q)\leq 2(1-t^d).
	\]
	Similarly,
	\[
	|Ew_r|(Q\setminus \overline{Q}_t)
	\leq
	|Ev_r|(Q\setminus \overline{Q}_t)+|Ev_0|(Q\setminus \overline{Q}_t)
	+
	\int_Q |\nabla\varphi|\,|v_r-v_0|\,\dif y
	\]
	and thus we also have
	\[
	\limsup_{r\downarrow 0}|Ew_r|(Q\setminus\overline{Q}_t)\leq 2(1-t^d).
	\]
	Using the scaling and growth properties of $\mathcal{F}_{*}$ we obtain
	\begin{align*}
	\frac{\mathcal{F}_{*}[u,Q(x_0,r)]}{|Eu|(Q(x_0,r))}
	&=
	\frac{\mathcal{F}_{*}[\alpha_rv_r,Q]}{\alpha_r} \\
	&\geq
	\frac{\mathcal{F}_{*}[\alpha_rw_r,\overline{Q}_t]}{\alpha_r} \\
	&=
	\frac{\mathcal{F}_{*}[\alpha_rw_r,Q]}{\alpha_r}
	-
	\frac{\mathcal{F}_{*}[\alpha_rw_r,Q\setminus\overline{Q}_t]}{\alpha_r} \\
	&\geq
	\frac{\mathcal{F}_{*}[\alpha_rw_r,Q]}{\alpha_r}
	-M\left(\alpha_r^{-1}|Q\setminus\overline{Q}_t|+|Ew_r|(Q\setminus\overline{Q}_t)\right).
	\end{align*}
	Since $\alpha_r\to+\infty$ as $r\downarrow 0$ we obtain
	\[
	\lim_{r\downarrow 0}\frac{\mathcal{F}_{*}[u,Q(x_0,r)]}{|Eu|(Q(x_0,r))}
	\geq
	\limsup_{r\downarrow 0}\frac{\mathcal{F}_{*}[\alpha_rw_r,Q]}{\alpha_r}
	-2M(1-t^d).
	\]
	By Lemma~\ref{lem:lower-est} in conjunction with the Lipschitz continuity of $f$ (cf.~\cite[Lemma~5.6]{Rindler18book}), we obtain
	\begin{align*}
	\mathcal{F}_{*}[\alpha_rw_r,Q]
	\geq
	f(\alpha_rEw_r(Q))
	\geq
	f(\alpha_r Ev_r(Q))-\alpha_rL|E(w_r-v_r)|(Q)
	\end{align*}
	for all $r>0$. Here $L>0$ denotes a Lipschitz constant of $f$.
	Therefore
	\begin{align*}
	\lim_{r\downarrow 0}\frac{\mathcal{F}_{*}[u,Q(x_0,r)]}{|Eu|(Q(x_0,r))}
	&\geq
	\limsup_{r\downarrow 0}\frac{f(\alpha_r Ev_r(Q))}{\alpha_r}-2(L+M)(1-t^d).
	\end{align*}
	Since
	\[
	Ev_r(Q)
	=
	\frac{Eu(Q(x_0,r))}{|Eu|(Q(x_0,r))}
	\to
	\frac{\dif E^su}{\dif|E^su|}(x_0)
	=
	a\odot b \quad \text{as} \quad r\downarrow 0,
	\]
	we obtain
	\begin{align*}
	\limsup_{r\downarrow 0}\frac{f(\alpha_rEv_r(Q))}{\alpha_r}
	=
	f^{\#}(a\odot b),
	\end{align*}
	We thus have
	\begin{align*}
	\lim_{r\downarrow 0}\frac{\mathcal{F}_{*}[u,Q(x_0,r)]}{|Eu|(Q(x_0,r))}
	&\geq f^{\#}(a\odot b)-2(L+M)(1-t^d).
	\end{align*}
	Letting $t\uparrow 1$ concludes the proof. \qedhere
\end{proof}

\bibliographystyle{plain}

\end{document}